\documentclass[reqno]{amsart}
\usepackage{amsaddr}
\usepackage[utf8]{inputenc}
\usepackage[T1]{fontenc}
\usepackage{amsmath, amsfonts, lineno, graphicx, tikz, xcolor, amsthm}
\usepackage[colorlinks=true, allcolors=blue]{hyperref}
\usepackage[scr=euler]{mathalpha}

\newtheorem{theorem}{Theorem}
\newtheorem{definition}{Definition}

\newcommand{\R}[0]{\mathcal{R}}
\newtheorem{corollary}{Corollary}

\begin{document}

\author[L.D.~Lu and T.~Vanzan]{Liu-Di Lu$^{{\lowercase\mathrm{a}}}$ and
		Tommaso Vanzan$^{{\lowercase\mathrm{b}}}$
		}

\address[Liu Di Lu]{
	$\,^{\lowercase\mathrm{a}}$ Centre for Mathematical Sciences, 
	Lund University\\ 
	Märkesbacken 4, 22362 Lund,
	Sweden}
	
\address[Tommaso Vanzan]{
	$\,^{\lowercase\mathrm{b}}$Dipartimento di Scienze Matematiche, 
	Politecnico di Torino \\
  	Corso Duca degli Abruzzi 24, 10129 Torino,  
	Italy}


\title[Weak scalability of time parallel Schwarz for parabolic OCPs]{Weak scalability of time parallel Schwarz methods for parabolic optimal control problems}


\keywords{parallel-in-time methods, parallel Schwarz, weak scalability, optimal control, Toeplitz matrices}

\subjclass[2020]{Primary
			 65M55; 
			  Secondary  49N10, 
			  65F08, 
			  35Q93  
			    }  %

\date{\today}


\begin{abstract}
Parabolic optimal control problems arise in numerous scientific and engineering applications. They typically lead to large-scale coupled forward–backward systems that cannot be treated with classical time-stepping schemes and are computationally expensive to solve. Therefore, parallel methods are essential to reduce the computational time required. In this work, we investigate a time domain decomposition approach, namely the time parallel Schwarz method, applied to parabolic optimal control problems. We analyze the convergence behavior and focus on the weak scalability property of this method as the number of time intervals increases. To characterize the spectral radius of the iteration matrix, we present two analysis techniques: the construction of a tailored matrix norm and the application of block Toeplitz matrix theory. Our analyses yield both nonasymptotic bounds on the spectral radius and an asymptotic characterization of the eigenvalues as the number of time intervals tends to infinity. 
Numerical experiments further confirm our theoretical findings and demonstrate the weak scalability of the time parallel Schwarz method. This work introduces the first theoretical tool for analyzing the weak scalability of time domain decomposition methods, and our results shed light on the suitability of our algorithm for large-scale simulations on modern high-performance computing architectures.
\end{abstract}

\maketitle

\section{Introduction}
Parabolic optimal control problems arise in a wide range of scientific and engineering applications, including diffusion-driven processes~\cite{Merger2017}, thermal regulation~\cite{Meinlschmidt2017}, environmental economics~\cite{augeraud2019distributed}, cancer treatment~\cite{schattler2015optimal} and inverse problems~\cite{Kimura1993}. Their mathematical formulation typically involves minimizing a cost functional subject to some time-dependent partial differential equation (PDE) constraints. After deriving the first-order optimality system, such problems lead to large scale coupled systems, whose efficient solution remains a central challenge in high-performance scientific computing. While spatial parallelism has been extensively exploited through domain decomposition and multigrid methods (see, e.g., \cite{benamou1999domain, lagnese2012domain, schoberl2011robust}), the time direction is often treated sequentially due to the intrinsic causality of evolution equations, that is, the solution at the current time depends on the solution at the previous instant. Therefore, it is not natural to parallelize the time direction. Nonetheless, as parallelization in space has reached saturation and massively parallel computing architectures have become increasingly available, these trends have motivated the development of parallel-in-time algorithms. We refer to~\cite{Gander2015, Gander2024} for a brief review.

In general, parallel algorithms often partition the global domain into many small subdomains, and local problems are solved iteratively while exchanging information with neighboring subdomains. When applied to optimal control problems governed by time-dependent PDEs, such strategies must address the forward-backward structure of the optimality system. Generally speaking, iterative parallel-in-time algorithms for parabolic optimal control problems can be classified into three main families: algorithms based on multiple shooting methods, e.g.~\cite{Gander2020, Maday2007, Mathew2010, Ulbrich2015}, algorithms based on multigrid methods, e.g.~\cite{Borzi2003, GonzalezAndrade2012, Gotschel2019, Gunther2019, Li2017}, and algorithms based on domain decomposition methods, e.g.~\cite{Barker2015, Ciaramella2024, Gander2016, Gander20242048, Gander20242588, Heinkenschloss2005, Mandal2018, Riahi2016}. Other parallel-in-time algorithms for such problems also exist, e.g., a modified matching-Schur complement preconditioner has recently been introduced in~\cite{Lin2024} to solve parabolic optimal control problems. Among these approaches, time domain decomposition methods proposed in~\cite{Gander2016, Gander20242048, Gander20242588} provide a novel framework that takes advantage of the forward-backward structure of the optimality system. In this work, we investigate the {\em weak scalability} of a nonoverlapping time domain decomposition method, namely a time parallel Schwarz algorithm, applied to parabolic optimal control problems. 

We recall that a parallel algorithm is said to be weakly scalable if it can solve larger and larger problems in a fixed amount of time with the number of processors increased proportionally. This is a weaker version of the so-called strong scalability, which requires that the acceleration generated by parallelization scales proportionally with the number of processors, that is, the time to solve a fixed-size problem decreases proportionally with the number of processors. Strong scalability is difficult to achieve since, in the limit, the work assigned to each processor becomes negligible, and the communication cost dominates.
The weak scalability property of an algorithm is however crucial for using high-performance computing architectures, as it measures its efficiency when the problem size and the number of processors simultaneously scale. This requires having both algorithmic convergence and parallel performance.
When specializing these definitions to iterative methods, we say that a domain decomposition method is weakly scalable if its contraction factor $\rho$ is uniformly bounded by a constant independent of the number of fixed-size subdomains $N$ and is strictly less than one. The relevance of the weak scalability property has been shown, for instance, in computational chemistry, where the authors apply domain decomposition methods to reduce the computational time in simulating long chains of molecules~\cite{Cances2013, Lipparini2013}. In~\cite{Ciaramella2017, Ciaramella2018}, the authors have studied the scalability of the parallel Schwarz method applied to elliptic problems with growing chains of fixed-sized subdomains, while~\cite{berrone2022weak} extended the methodology to the topological structures appearing in discrete fracture networks.
The authors in~\cite{Chaouqui2018} further investigate the weak scalability of four classical, one-level, domain decomposition methods applied to elliptic problems in both one- and two-dimensional settings. Their results showed that the parallel Schwarz method applied to a two-dimensional Poisson equation is weakly scalable, while it is not for a one-dimensional problem, where further coarse corrections need to be added to recover scalability. An extension to time-harmonic problems is available in~\cite{bootland2021analysis}.

In the current work, we focus on solving the first-order optimality system of parabolic optimal control problems using the time parallel Schwarz method.  We investigate three main research questions:
\begin{itemize}
\item As the time direction is of dimension one, is the time parallel Schwarz method weakly scalable?

\item How does the weak scalability depend on the fixed-size time interval length and problem parameters?

\item Can we characterize the convergence behavior of the time parallel Schwarz method?
\end{itemize}
In our context, the weak scalability regime corresponds to simulating and controlling the physical system over longer and longer time intervals.
To answer these questions, we provide a detailed convergence analysis. After a semi-discretization in space, we characterize the asymptotic behavior of the time parallel Schwarz algorithm using two different techniques: (i) construct a special matrix norm, (ii) use block Toeplitz matrix theory~\cite{Bottcher1997, Trefethen2005}. These approaches allow us to derive two complementary results that not only characterize the spectral radius of the iteration matrix, but also describe the asymptotic distribution of the spectrum as the number of time intervals tends to infinity.

The rest of the manuscript is organized as follows. Section~\ref{sec:2} introduces the parabolic optimal control problem, its optimality system, and defines the time parallel Schwarz algorithm. Section~\ref{sec:3} presents the weak scalability analysis using two different techniques, and derives an estimate for the spectral radius of the iteration matrix. Section~\ref{sec:4} shows numerical experiments to confirm our theoretical results and illustrates the weak scalability of the algorithm. Section~\ref{sec:5} summarizes the main findings and outlines directions for future work.

\section{Model problem}\label{sec:2}
In this section, we define our linear quadratic parabolic control problem. Consider a final time $T > 0$ and a bounded open set $\Omega \subset \mathbb{R}^d$, $d=1, 2, 3$, with Lipschitz boundary $\partial\Omega$. The space-time domain is defined by $Q := \Omega\times (0, T)$. We are interested in solving the following distributed control problem: for a given target $\hat{y} \in L^2(Q)$ and a penalization parameter $\nu >0$, we minimize the cost functional
\begin{equation}\label{eq:J}
J(y, u) := \frac{1}{2} \|y -\hat{y}\|^2_{L^2(Q)}  + \frac{\nu}{2}\|u\|^2_{L^2(Q)},
\end{equation}
subject to the linear parabolic equation\footnote{Our analysis covers the general state equation $\partial_t y -\mathcal{L}(y) =u $, provided that a spatial discretization of $\mathcal{L}$ leads to a diagonalizable matrix $A$ with real and positive eigenvalues, see Section \ref{sec:eigendecomposition}.} 
\begin{equation}\label{eq:heat}
\partial_t y - \Delta y = u \ \text{ in } Q, 
\quad
y = g  \ \text{ on } \partial\Omega\times (0, T),
\quad
y = y_{0} \ \text{ on } \Omega\times\{0\},
\end{equation}
with $y_0\in L^2(\Omega)$ a given initial condition and $g$ a sufficiently regular function. Deriving the optimality conditions of problem~\eqref{eq:J}-\eqref{eq:heat} using the Lagrange multipliers technique, we obtain the reduced first-order optimality system:
\begin{equation}\label{eq:reduced}
\begin{aligned}
\partial_t y - \Delta y &= \nu^{-1}p& \text{ in }& Q,& 
y &= g& \text{ on }& \partial\Omega\times (0, T),&
y &= y_{0}&  \text{ on }& \Omega\times\{0\},&  
\\
\partial_t p + \Delta p &= y - \hat y& \text{ in }& Q,& 
p &= 0& \text{ on }& \partial\Omega\times (0, T),&
p &= 0&  \text{ on }& \Omega\times\{T\},
\end{aligned}
\end{equation} 
$p$ being the Lagrange multiplier, and with the optimal control given by $u=p/\nu$. We refer to~\cite[Chapter 3]{Troltzsch2010} for a detailed derivation of such an optimality system. The goal of our study is to solve the reduced first-order optimality system~\eqref{eq:reduced} using time domain decomposition methods. In particular, we are interested in the behavior of these methods when the number of time intervals $N$ grows while their size is fixed, representing a larger and larger final time $T$.

\subsection{Time domain decomposition}
Let $N\in\mathbb{N}$ and consider a set of time intervals $\left\{(t_{n-1}, t_n)\right\}_{n=1}^N$ shown in Figure~\ref{fig:decomposition}. 
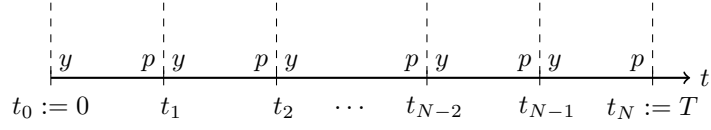
\begin{figure}
\begin{center}
\begin{tikzpicture}
\draw [thick, ->] (0, 0) -- (8.5, 0) node[anchor = west] {$t$};
\foreach \x in {0, 1.5, 3, 5, 6.5, 8}{\draw (\x, 0.1) -- (\x, -0.1);}
\node at (0, -0.4) {$t_0 := 0$};
\node at (1.6, -0.4) {$t_1$};
\node at (3.1, -0.4) {$t_2$};
\node at (4, -0.4) {\ldots};
\node at (5.1, -0.4) {$t_{N-2}$};
\node at (6.6, -0.4) {$t_{N-1}$};
\node at (8, -0.4) {$t_{N} := T$};
\foreach \x in {0, 1.5, 3, 5, 6.5, 8}{\draw[dashed] (\x, 1) -- (\x, 0);}
\foreach \x in {0.2, 1.7, 3.2, 5.2, 6.7} {\node at (\x, 0.2) {$y$};}
\foreach \x in {1.3, 2.8, 4.8, 6.3, 7.8} {\node at (\x, 0.2) {$p$};}
\end{tikzpicture}
\end{center}
\caption{Illustration of $N$ time intervals with fixed size $\Delta t$.}
\label{fig:decomposition}
\end{figure}
Each time interval has a finite, fixed size $\Delta t$, $t_n := n\Delta t$ for $n \in \{0, \ldots, N\}$, and $T=N\Delta t$. We denote the space-time subdomain by $Q_n := \Omega\times(t_{n-1}, t_n)$ for $n \in\{1, \ldots, N\}$ and the interface by $\Sigma_n := \Omega\times \{t_n\}$ for $n = 0, \ldots, N$. The optimality system~\eqref{eq:reduced} has a forward-backward structure, where the state variable $y$ propagates forward in time with an initial condition $y(\cdot, 0) = y_0(\cdot)$ at $\Sigma_0$, and the adjoint variable propagates backward in time with a final condition $p(\cdot, T) = 0$ at $\Sigma_N$. To retain the same forward-backward structure for the restriction of the system~\eqref{eq:reduced} in each subdomain $Q_n$, one needs to impose an "initial" condition for $y$ at the interface $\Sigma_{n-1}$ and a "final" condition $p$ at the interface $\Sigma_n$, $n\in\{1,\dots,N\}$, as illustrated in Figure~\ref{fig:decomposition}. We define the functions $y_n$ and $p_n$ as the restriction of $y$ and $p$ to $Q_n$. Then, the restricted system in $Q_n$ using the same forward-backward structure of~\eqref{eq:reduced} can be formulated as follows: for $n \in\{2, \ldots, N-1\}$, the pair $(y_n, p_n)$ satisfies
\begin{equation}\label{eq:Q2-QN-1}
\begin{aligned}
\partial_t y_n - \Delta y_n &= \nu^{-1}p_n \ \text{ in } Q_n,& 
y_n &= g \ \text{ on } \partial\Omega\times (0, T),& 
y_n &= y_{n-1} \ \text{ on } \Sigma_{n-1},
\\
\partial_t p_n + \Delta p_n&= y_n - \hat{y}_n  \ \text{ in } Q_n,&
p_n &= 0 \ \text{ on } \partial\Omega\times (0, T),&
p_n &= p_{n+1} \ \text{ on } \Sigma_n,
\end{aligned}
\end{equation}
where the conditions at $\Sigma_{n-1}$ and $\Sigma_n$ describe the interaction of the $n$-th subdomain with its neighbor subdomains $n-1$ and $n+1$, and the function $\hat{y}_n$ is the restriction of the target function $\hat y$ to $Q_n$. The pair $(y_1, p_1)$ of the first subdomain $Q_1$ solves
\begin{equation}\label{eq:Q1}
\begin{aligned}
\partial_t y_1 - \Delta y_1 &= \nu^{-1}p_1 \ \text{ in } Q_1,&
y_1 &= g \ \text{ on } \partial\Omega\times (0, T),&
y_1 &= y_{0} \ \text{ on } \Sigma_{0}, 
\\
\partial_t p_1 + \Delta p_1 &= y_1 - \hat{y}_1  \ \text{ in } Q_1,&
p_1 &= 0 \ \text{ on } \partial\Omega\times (0, T),&
p_1 &= p_{2} \ \text{ on } \Sigma_1,
\end{aligned}
\end{equation}
where $y_0$ is the given initial condition. The pair $(y_N, p_N)$ of the last subdomain solves
\begin{equation}\label{eq:QN}
\begin{aligned}
\partial_t y_N - \Delta y_N &= \nu^{-1}p_N \ \text{ in } Q_N,& 
y_N &= g \ \text{ on } \partial\Omega\times (0, T),&
y_N &= y_{N-1} \ \text{ on } \Sigma_{N-1}, 
\\
\partial_t p_N + \Delta p_N &= y_N - \hat{y}_N \ \text{ in } Q_N,&
p_N &= 0 \ \text{ on } \partial\Omega\times (0, T),&
p_N &= 0 \ \text{ on } \Sigma_{N},
\end{aligned}
\end{equation}
where $p_N|_{\Sigma_N} = 0$ is the given final condition for $p$ in~\eqref{eq:reduced}.

Note that the choice $p_n = p_{n+1}$ on $\Sigma_n$, $n \in\{ 1, \ldots, N-1\}$, imposes the continuity in time of the adjoint variable $p$ at each interface. Similarly, the choice $y_n = y_{n-1}$ on $\Sigma_{n-1}$, $n \in\{ 1, \ldots, N\}$, imposes the continuity in time of the state variable $y$.

\subsection{Time parallel Schwarz method}
The classical parallel Schwarz method (PSM) was introduced by Lions in~\cite{lions1988schwarz} to generalize the alternating Schwarz method to a parallel setting. Here, we are interested in applying the PSM to solve~\eqref{eq:Q2-QN-1}-\eqref{eq:QN} with a \textit{non-overlapping} decomposition in time. In particular, we analyze the convergence behavior of the time PSM for a growing number of fixed-sized space-time subdomains and investigate its weak scalability. Consider the problem~\eqref{eq:Q2-QN-1}-\eqref{eq:QN} and some initial guesses $(y_1^0, p_1^0), (y_{2}^0, p_{2}^0), \ldots, (y_N^0, p_N^0)$. Since we are interested in the time decomposition, we omit hereafter to specify the boundary conditions $y_n = g$ and $p_n = 0$ on $\partial\Omega\times (0, T)$ for $n \in \{ 1, \ldots, N\}$. For the iteration index $\ell = 1, 2, \ldots$, the time PSM defines the approximation sequences $\{y_n^\ell, p_n^\ell\}_\ell$ by solving
\begin{equation}\label{eq:PSQ2-QN-1}
\begin{aligned}
\partial_t y_n^\ell - \Delta y_n^\ell &= \nu^{-1}p_n^\ell&   \text{ in }& Q_n,& 
y_n^\ell &= y_{n-1}^{\ell-1}& \text{ on }& \Sigma_{n-1}, 
\\
\partial_t p_n^\ell + \Delta p_n^\ell &= y_n^\ell - \hat{y}_n& \text{ in }& Q_n,&
p_n^\ell &= p_{n+1}^{\ell-1}& \text{ on }& \Sigma_n,&
\end{aligned}
\end{equation}
for $n \in\{ 2, \ldots, N-1\}$. In the first time interval, $\{y_1^\ell, p_1^\ell\}_\ell$ satisfy 
\begin{equation}\label{eq:PSQ1}
\begin{aligned}
\partial_t y_1^\ell - \Delta y_1^\ell &= \nu^{-1}p_1^\ell& \text{ in }& Q_1,& 
y_1^\ell &= y_{0} & \text{ on }& \Sigma_{0},& \\
\partial_t p_1^\ell + \Delta p_1^\ell &= y_1^\ell - \hat{y}_1& \text{ in }& Q_1,& 
p_1^\ell &= p_{2}^{\ell-1}& \text{ on }& \Sigma_1,&
\end{aligned}
\end{equation}
while in the last interval $\{y_N^\ell, p_N^\ell\}_\ell$ solves
\begin{equation}\label{eq:PSQN}
\begin{aligned}
\partial_t y_N^\ell - \Delta y_N^\ell &= \nu^{-1}p_N^\ell& \text{ in }& Q_N,& 
y_N^\ell &= y_{N-1}^{\ell-1}& \text{ on }& \Sigma_{N-1},& \\
\partial_t p_N^\ell + \Delta p_N^\ell &= y_N^\ell - \hat{y}_N& \text{ in }& Q_N,& 
p_N^\ell &= 0& \text{ on }& \Sigma_N.&
\end{aligned}
\end{equation}
In each subdomain $Q_n$, $n \in\{1, \ldots, N\}$, the algorithm~\eqref{eq:PSQ2-QN-1}-\eqref{eq:PSQN} uses the information at the previous iteration to update the value of $(y_n, p_n)$ on the interface. The state variable $y_n$ takes information coming from the left neighbouring time interval at the previous iteration, and the adjoint variable $p_n$ takes information coming from the right neighbouring time interval at the previous iteration. Therefore, each subproblem can be solved independently and in parallel at each iteration.

\subsection{Eigendecomposition}\label{sec:eigendecomposition}
To analyze the convergence of the time PSM algorithm~\eqref{eq:PSQ2-QN-1}-\eqref{eq:PSQN}, we introduce the errors $e_{y_n^\ell} := y_n - y_n^\ell$ and $e_{p_n^\ell} := p_n - p_n^\ell$. Substracting the exact solutions of~\eqref{eq:Q2-QN-1}-\eqref{eq:QN} from the approximations~\eqref{eq:PSQ2-QN-1}-\eqref{eq:PSQN}, the errors then satisfy 
\begin{equation}\label{eq:PSErrQ2-QN-1}
\begin{aligned}
\partial_t e_{y_n^\ell} - \Delta e_{y_n^\ell} &= \nu^{-1}e_{p_n^\ell}&   \text{ in }& Q_n,& 
e_{y_n^\ell} &= e_{y_{n-1}^{\ell-1}}& \text{ on }& \Sigma_{n-1}, 
\\
\partial_t e_{p_n^\ell} + \Delta e_{p_n^\ell} &= e_{y_n^\ell} & \text{ in }& Q_n,&
e_{p_n^\ell} &= e_{p_{n+1}^{\ell-1}}& \text{ on }& \Sigma_n,&
\end{aligned}
\end{equation}
for $n \in\{ 2, \ldots, N-1\}$, and
\begin{equation}\label{eq:PSErrQ1}
\begin{aligned}
\partial_t e_{y_1^\ell} - \Delta e_{y_1^\ell} &= \nu^{-1}e_{p_1^\ell}& \text{ in }& Q_1,& 
e_{y_1^\ell} &= 0 & \text{ on }& \Sigma_{0},& \\
\partial_t e_{p_1^\ell} + \Delta e_{p_1^\ell} &= e_{y_1^\ell} & \text{ in }& Q_1,& 
e_{p_1^\ell} &= e_{p_{2}^{\ell-1}}& \text{ on }& \Sigma_1,&
\end{aligned}
\end{equation}
for $n=1$, and
\begin{equation}\label{eq:PSErrQN}
\begin{aligned}
\partial_t e_{y_N^\ell} - \Delta e_{y_N^\ell} &= \nu^{-1}e_{p_N^\ell}& \text{ in }& Q_N,& 
e_{y_N^\ell} &= e_{y_{N-1}^{\ell-1}}& \text{ on }& \Sigma_{N-1},& \\
\partial_t e_{p_N^\ell} + \Delta e_{p_N^\ell} &= e_{y_N^\ell} & \text{ in }& Q_N,& 
e_{p_N^\ell} &= 0& \text{ on }& \Sigma_N,&
\end{aligned}
\end{equation}
for $n=N$. The boundary conditions on $\partial \Omega \times (0, T)$ become $e_{y_n^\ell} = 0$ and $e_{p_n^\ell} = 0$ for all $n \in \{1, \ldots, N\}$.

Since our goal is to investigate the scalability in time, we introduce a spatial discretization $-\Delta \approx A \in \mathbb{R}^{M\times M}$, where $M$ corresponds to the number of degrees of freedom. For instance, one can use a centered finite difference or a finite element discretization in space. In these cases, the matrix $A$ is diagonalizable with an invertible matrix $P$ such that $P^{-1}AP = \text{diag}(\lambda_1, \ldots, \lambda_M)$ with $0<\lambda_1<\ldots<\lambda_M$. Well-known results (see, e.g., ~\cite[Chapter 6.3.2]{quarteroni1994numerical}) fully characterize the asymptotic behavior of $\{\lambda_m\}_{m=1}^M$ with respect to the mesh size: provided that the mesh is quasi-uniform, there exist two constants $\underline{C}$ and $\overline{C}$ such that $\underline{C} h^d\leq\lambda_m\leq \overline{C}h^{d-2}$ for any $m\in\{1,\dots,M\}$. 
This decomposition allows us to focus only on the time variable.

Applying first the spatial discretization and then diagonalization to~\eqref{eq:PSErrQ2-QN-1}-\eqref{eq:PSErrQN}, we obtain for each eigenvalue $\lambda_m$ that, 
\begin{equation}\label{eq:PSErrQ2-QN-1ODE}
\begin{aligned}
\dot z_{n, m}^\ell + \lambda_m z_{n, m}^\ell &= \nu^{-1}q_{n, m}^\ell& \text{ in }& (t_{n-1}, t_n),& 
z_{n, m}^\ell(t_{n-1}) &= z_{n-1, m}^{\ell-1}(t_{n-1}),
\\
\dot q_{n, m}^\ell - \lambda_m q_{n, m}^\ell &= z_{n, m}^\ell& \text{ in }& (t_{n-1}, t_n),& 
q_{n, m}^\ell(t_n) &= q_{n+1, m}^{\ell-1}(t_n),
\end{aligned}
\end{equation}
for $n \in\{ 2, \ldots, N-1\}$, and
\begin{equation}\label{eq:PSErrQ1ODE}
\begin{aligned}
\dot z_{1, m}^\ell + \lambda_m z_{1, m}^\ell &= \nu^{-1}q_{1, m}^\ell& \text{ in }& (t_0, t_1),& 
z_{1, m}^\ell(t_0) &= 0,
\\
\dot q_{1, m}^\ell - \lambda_m q_{1, m}^\ell &= z_{1, m}^\ell& \text{ in }& (t_0, t_1),&
q_{1, m}^\ell(t_1) &= q_{1, m}^{\ell-1}(t_1),&
\end{aligned}
\end{equation}
for $n = 1$, and
\begin{equation}\label{eq:PSErrQNODE}
\begin{aligned}
\dot z_{N, m}^\ell + \lambda_m z_{N, m}^\ell &= \nu^{-1}q_{N, m}^\ell& \text{ in }& (t_{N-1}, t_N),&
z_{N, m}^\ell(t_{N-1}) &= z_{N-1, m}^{\ell-1}(t_{N-1}), 
\\
\dot q_{N, m}^\ell - \lambda_m q_{N, m}^\ell &= z_{N, m}^\ell & \text{ in }& (t_{N-1}, t_N),&
q_{N, m}^\ell(t_N) &= 0,
\end{aligned}
\end{equation}
for $n=N$. Here, $\boldsymbol{z}_n^\ell := P^{-1}\boldsymbol{e}_{y_n^\ell}$, $\boldsymbol{q}_n^\ell := P^{-1}\boldsymbol{e}_{p_n^\ell}$, and $z_{n, m}^\ell$, $q_{n, m}^\ell$ are the $m$-th components of the vectors $\boldsymbol{z}_n^\ell$, $\boldsymbol{q}_n^\ell$. The dot denotes the common notation of the time derivative. The iterations~\eqref{eq:PSErrQ2-QN-1ODE}-\eqref{eq:PSErrQNODE} are systems of coupled first-order ODEs for each eigenvalue $\lambda_m$, and there are in total $M$ independent such systems. In particular, we have the following two relations between the variables $z_{n, m}^\ell$ and $q_{n, m}^\ell$,
\[q_{n, m}^\ell = \nu\left(\dot z_{n, m}^\ell + \lambda_m z_{n, m}^\ell\right),
\quad
z_{n, m}^\ell = \dot q_{n, m}^\ell - \lambda_m q_{n, m}^\ell,\]
for all $n\in\{1, \ldots, N\}$ and all $m\in\{1, \ldots, M\}$. Eliminating $q_{n, m}^\ell$ by $z_{n, m}^\ell$, one can transform the coupled first-order ODE systems~\eqref{eq:PSErrQ2-QN-1ODE}-\eqref{eq:PSErrQNODE} into second-order ODEs of $z_{n, m}^\ell$ as
\begin{equation}\label{eq:PSErrQ2-QN-1ODE2}
\begin{aligned}
\ddot z_{n, m}^\ell - \sigma_m^2 z_{n, m}^\ell &= 0 \quad \text{ in } (t_{n-1}, t_n), 
\\
z_{n, m}^\ell(t_{n-1}) &= z_{n-1, m}^{\ell-1}(t_{n-1}), 
\\
\dot z_{n, m}^\ell(t_n) + \lambda_m z_{n, m}^\ell (t_n) &= \dot z_{n+1, m}^{\ell-1}(t_n) + \lambda_m z_{n+1, m}^{\ell-1} (t_n),
\end{aligned}
\end{equation}
for $n \in \{2, \ldots, N-1\}$, and
\begin{equation}\label{eq:PSErrQ1ODE2}
\begin{aligned}
\ddot z_{1, m}^\ell - \sigma_m^2 z_{1, m}^\ell &= 0 \quad \text{ in } (t_0, t_1), \\
z_{1, m}^\ell(t_0) &= 0, \\
\dot z_{1, m}^\ell(t_1) + \lambda_m z_{1, m}^\ell(t_1) &= \dot z_{2, m}^{\ell-1}(t_1) + \lambda_m z_{2, m}^{\ell-1}(t_1),
\end{aligned}
\end{equation}
for $n=1$, and for $n=N$
\begin{equation}\label{eq:PSErrQNODE2}
\begin{aligned}
\ddot z_{N, m}^\ell - \sigma_m^2 z_{N, m}^\ell &= 0 \quad \text{ in } (t_{N-1}, t_N), \\
z_{N, m}^\ell(t_{N-1}) &= z_{N-1, m}^{\ell-1}(t_{N-1}), \\
\dot z_{N, m}^\ell(t_N) + \lambda_m z_{N, m}^\ell(t_N) &= 0,
\end{aligned}
\end{equation}
where $\sigma_m := \sqrt{\lambda_m^2 + 1/\nu}$. Note that by doing so, we transform a Dirichlet condition $q_{n, m}^\ell(t_n) = q_{n+1, m}^{\ell-1}(t_n)$ into a Robin type condition in time $\dot z_{n, m}^\ell(t_n) + \lambda_m z_{n, m}^\ell (t_n) = \dot z_{n+1, m}^{\ell-1}(t_n) + \lambda_m z_{n+1, m}^{\ell-1} (t_n)$; and the forward-backward structure in~\eqref{eq:PSErrQ2-QN-1ODE}-\eqref{eq:PSErrQNODE} disappears in~\eqref{eq:PSErrQ2-QN-1ODE2}-\eqref{eq:PSErrQNODE2}. 
Alternatively, one can eliminate $z_{n, m}^\ell$ by $q_{n, m}^\ell$ and transform~\eqref{eq:PSErrQ2-QN-1ODE}-\eqref{eq:PSErrQNODE} into second-order ODEs of  $q_{n, m}^\ell$. As they are all equivalent, we consider~\eqref{eq:PSErrQ2-QN-1ODE2}-\eqref{eq:PSErrQNODE2} to analyze the convergence.

\subsection{Iteration matrix}
A general solution of these second-order ODEs~\eqref{eq:PSErrQ2-QN-1ODE2}-\eqref{eq:PSErrQNODE2} is given by
\[z_{n, m}^\ell(t) = A_{n, m}^\ell \sinh(\sigma_m t) + B_{n, m}^\ell \cosh(\sigma_m t),\]
where $A_{n, m}^\ell$, $B_{n, m}^\ell$ are coefficients determined by the initial and final conditions in each time interval $(t_{n-1}, t_n)$. We also specify the Robin term, which appears frequently, 
\begin{equation*}
\begin{split}
\dot z_{n, m}^\ell(t) + \lambda_m z_{n, m}^\ell(t) = &A_{n, m}^\ell\big(\sigma_m\cosh(\sigma_m t) + \lambda_m\sinh(\sigma_m t)\big)\\ 
&+ B_{n, m}^\ell\big(\sigma_m\sinh(\sigma_m t) + \lambda_m\cosh(\sigma_m t)\big).
\end{split}
\end{equation*}
To simplify notations, we denote the Dirichlet transmission condition by $D_{n, m}^\ell := z_{n, m}^\ell(t_n)$ and the Robin transmission condition by $R_{n, m}^\ell := \dot{z}_{n, m}^\ell(t_{n-1}) + \lambda_m z_{n, m}^\ell(t_{n-1})$. We also use the shorthand notations $s_{n, m} := \sinh(\sigma_m t_n)$ and $c_{n, m} := \cosh(\sigma_m t_n)$.

For the middle time intervals $n \in\{ 2, \ldots, N-1\}$, the two transmission conditions in~\eqref{eq:PSErrQ2-QN-1ODE2} give
\begin{equation*}
\begin{aligned}
A_{n, m}^\ell s_{n-1, m} + B_{n, m}^\ell c_{n-1, m}  &= D_{n-1, m}^{\ell-1}, \\
A_{n, m}^\ell (\sigma_mc_{n, m} + \lambda_ms_{n, m}) + B_{n, m}^\ell (\sigma_ms_{n, m} +  \lambda_mc_{n, m}) &= R_{n+1, m}^{\ell-1}.
\end{aligned}
\end{equation*}
A direct computation leads to the two coefficients in the solutions $z_{n, m}^\ell(t)$ of~\eqref{eq:PSErrQ2-QN-1ODE2},
\[\begin{aligned}
A_{n, m}^\ell &= \frac{ D_{n-1, m}^{\ell-1}(\sigma_ms_{n, m} +  \lambda_mc_{n, m}) - R_{n+1, m}^{\ell-1}c_{n-1, m} }
{ s_{n-1, m}(\sigma_ms_{n, m} +  \lambda_mc_{n, m})  - c_{n-1, m}(\sigma_mc_{n, m} + \lambda_ms_{n, m}) },
\\
B_{n, m}^\ell &= \frac{ R_{n+1, m}^{\ell-1}s_{n-1, m} - D_{n-1, m}^{\ell-1} (\sigma_mc_{n, m} + \lambda_ms_{n, m}) }
{ s_{n-1, m} (\sigma_ms_{n, m} +  \lambda_mc_{n, m}) - c_{n-1, m}  (\sigma_mc_{n, m} + \lambda_ms_{n, m}) }.
\end{aligned}\]
Evaluating then $\dot z_{n, m}^\ell(t) + \lambda_m z_{n, m}^\ell(t)$ at $t_{n-1}$ gives 
\[\begin{aligned}
& R_{n, m}^\ell = \dot{z}_{n, m}^\ell(t_{n-1}) + \lambda_m z_{n, m}^\ell(t_{n-1})
\\
= &\frac{ D_{j-1, i}^{\ell-1}(\sigma_ms_{n, m} +  \lambda_mc_{n, m}) - R_{n+1, m}^{\ell-1}c_{n-1, m} }
{ s_{n-1, m}(\sigma_ms_{n, m} +  \lambda_mc_{n, m})  - c_{n-1, m}(\sigma_mc_{n, m} + \lambda_ms_{n, m}) }(\sigma_m c_{n-1, m} + \lambda_ms_{n-1, m}) 
\\
&+ \frac{ R_{n+1, m}^{\ell-1}s_{n-1, m} - D_{n-1, m}^{\ell-1} (\sigma_mc_{n, m} + \lambda_ms_{n, m}) }
{ s_{n-1, m} (\sigma_ms_{n, m} +  \lambda_mc_{n, m}) - c_{n-1, m}  (\sigma_mc_{n, m} + \lambda_ms_{n, m}) }(\sigma_m s_{n-1, m}  + \lambda_mc_{n-1, m})
\\
= &\frac{ (\sigma_m^2 - \lambda_m^2)(c_{n-1, m}s_{n, m} - s_{n-1, m}c_{n, m})D_{n-1, m}^{\ell-1} + \sigma_m(s_{n-1, m}^2  - c_{n-1, m}^2)R_{n+1, m}^{\ell-1} }
{ s_{n-1, m}(\sigma_ms_{n, m} +  \lambda_mc_{n, m})  - c_{n-1, m}(\sigma_mc_{n, m} + \lambda_ms_{n, m}) }.
\end{aligned}\]
Using properties of the sums and differences of arguments for hyperbolic functions, the denominator can be simplified to
\[\begin{aligned}
&s_{n-1, m}(\sigma_ms_{n, m} +  \lambda_mc_{n, m})  - c_{n-1, m}(\sigma_mc_{n, m} + \lambda_ms_{n, m}) \\
= \; &\sigma_m(s_{n-1, m}s_{n, m} - c_{n-1, m}c_{n, m}) + \lambda_m(s_{n-1, m}c_{n, m} - c_{n-1, m}s_{n, m}) \\
= \; &\sigma_m\left(\sinh(\sigma_m t_{n-1})\sinh(\sigma_m t_n) - \cosh(\sigma_m t_{n-1})\cosh(\sigma_m t_n) \right) \\
&+ \lambda_m\left(\sinh(\sigma_m t_{n-1})\cosh(\sigma_m t_n) - \cosh(\sigma_m t_{n-1})\sinh(\sigma_m t_n)\right) \\
= \; &-\sigma_m\cosh\left(\sigma_m (t_n - t_{n-1})\right) + \lambda_m\sinh\left(\sigma_m (t_{n-1} - t_n)\right).
\end{aligned}\]
Similarly the numerators become $(\sigma_m^2 - \lambda_m^2)(c_{n-1, m}s_{n, m} - s_{n-1, m}c_{n, m}) =\sinh(\sigma_m (t_n - t_{n-1})) / \nu$, and $\sigma_m(s_{n-1, m}^2  - c_{n-1, m}^2) = -\sigma_m$ by analogy with the Pythagorean trigonometric identity.  Recalling that $\Delta t = t_n - t_{n-1}$, we find
\begin{equation}\label{eq:Rn}
R_{n, m}^\ell = \dot{z}_{n, m}^\ell(t_{n-1}) + \lambda_m z_{n, m}^\ell(t_{n-1}) = C_{1, m}(\Delta t) D_{n-1, m}^{\ell-1} + C_{2, m}(\Delta t) R_{n+1, m}^{\ell-1},
\end{equation}
with
\begin{equation}\label{eq:expression_C1C2}
\begin{aligned}
C_{1, m}(\Delta t) &:= \frac{-\nu^{-1}\sinh(\sigma_m \Delta t)}{\sigma_m\cosh(\sigma_m \Delta t) + \lambda_m\sinh (\sigma_m \Delta t)},\\
C_{2, m}(\Delta t) &:= \frac{\sigma_m}{\sigma_m\cosh(\sigma_m\Delta t) + \lambda_m\sinh (\sigma_m\Delta t)}.
\end{aligned}
\end{equation}
Note that $0<C_{2, m}(\Delta t)<1$ for any positive $\lambda_m$, $\nu$ and $\Delta t$. On the other hand, evaluating $z_{n, m}^\ell(t)$ at $t_n$ gives
\[\begin{aligned}
D_{n, m}^\ell = &z_{n, m}^\ell(t_n)
\\
= &\frac{ D_{n-1, m}^{\ell-1}(\sigma_ms_{n, m} +  \lambda_mc_{n, m}) - R_{n+1, m}^{\ell-1}c_{n-1, m} }
{ s_{n-1, m}(\sigma_ms_{n, m} +  \lambda_mc_{n, m})  - c_{n-1, m}(\sigma_mc_{n, m} + \lambda_ms_{n, m}) } s_{n, m} \\
&+ \frac{ R_{n+1, m}^{\ell-1}s_{n-1, m} - D_{n-1, m}^{\ell-1} (\sigma_mc_{n, m} + \lambda_ms_{n, m}) }
{ s_{n-1, m} (\sigma_ms_{n, m} +  \lambda_mc_{n, m}) - c_{n-1, m}  (\sigma_mc_{n, m} + \lambda_ms_{n, m}) } c_{n, m}
\\
= & \frac{ \sigma_m(s_{n, m}^2- c_{n, m}^2)D_{n-1, m}^{\ell-1} + (s_{n-1, m}c_{n, m} - c_{n-1, m}s_{n, m})R_{n+1, m}^{\ell-1} }
{ s_{n-1, m}(\sigma_ms_{n, m} +  \lambda_mc_{n, m})  - c_{n-1, m}(\sigma_mc_{n, m} + \lambda_ms_{n, m}) }.
\end{aligned}\]
Once again, we have the same denominator as in $R_{n, m}^\ell$, and the numerators are similar to those in $R_{n, m}^\ell$. Using the same computations as for $R_{n, m}^\ell$, we find
\begin{equation}\label{eq:Dn}
D_{n, m}^\ell = z_{n, m}^\ell(t_n) = C_{2, m}(\Delta t) D_{n-1, m}^{\ell-1} - \nu C_{1, m}(\Delta t) R_{n+1, m}^{\ell-1}.
\end{equation}

For the first time interval $n = 1$, the initial condition in~\eqref{eq:PSErrQ1ODE2}, $z_{1, m}^\ell(t_0) = 0$, gives directly $z_{1, m}^\ell(t) = A_{1, m}^\ell \sinh(\sigma_m t)$. The transmission condition $\dot z_{1, m}^\ell(t_1) + \lambda_m z_{1, m}^\ell(t_1) = R_{2, m}^{\ell-1}$ then gives $A_{1, m}^\ell = R_{2, m}^{\ell-1}/(\sigma_m c_{1, m} + \lambda_m s_{1, m})$.
Evaluating $z_{1, m}^\ell(t)$ at $t_1$ gives
\begin{equation}\label{eq:D1}
D_{1, m}^\ell = z_{1, m}^\ell(t_1) = \frac{ s_{1, m} }{ \sigma_mc_{1, m} + \lambda_ms_{1, m} } R_{2, m}^{\ell-1} = -\nu C_{1, m}(\Delta t) R_{2, m}^{\ell-1},
\end{equation}
since $t_1 = \Delta t$, then $s_{1, m} = \sinh(\sigma_m \Delta t)$ and $c_{1, m} = \cosh(\sigma_m \Delta t) $. 

For the last time interval $n=N$, the final condition in~\eqref{eq:PSErrQNODE2} $\dot z_{N, m}^\ell(t_N) + \lambda_m z_{N, m}^\ell(t_N) = 0$ gives $B_{N, m}^\ell = -A_{N, m}^\ell (\sigma_mc_{N, m} + \lambda_ms_{N, m})/(\sigma_ms_{N, m} +  \lambda_mc_{N, m})$.
The transmission condition $z_{N, m}^\ell(t_{N-1}) =D_{N-1, m}^{\ell-1}$ then gives
\[\begin{aligned}
A_{N, m}^\ell = &\frac{\sigma_ms_{N, m} +  \lambda_mc_{N, m}}{s_{N-1, m}(\sigma_ms_{N, m} +  \lambda_mc_{N, m}) - c_{N-1, m}(\sigma_mc_{N, m} + \lambda_ms_{N, m})} D_{N-1, m}^{\ell-1}
\\
= &-\frac{\sigma_ms_{N, m} +  \lambda_mc_{N, m}}{\sigma_m\cosh(\sigma_m\Delta t) + \lambda_m\sinh (\sigma_m\Delta t)}D_{N-1, m}^{\ell-1}.
\end{aligned}\]
Evaluating $\dot z_{N, m}^\ell(t) + \lambda_m z_{N, m}^\ell(t)$ at $t_{N-1}$ gives
\begin{equation}\label{eq:RN}
\begin{aligned}
 R_{N, m}^\ell  =\; & \dot z_{N, m}^\ell(t_{N-1}) + \lambda_m z_{N, m}^\ell(t_{N-1})
\\
=\; &\frac{-(\sigma_ms_{N, m} +  \lambda_mc_{N, m})(\sigma_mc_{N-1, m} + \lambda_ms_{N-1, m})}{ \sigma_m\cosh(\sigma_m\Delta t) + \lambda_m\sinh (\sigma_m\Delta t) } D_{N-1, m}^{\ell-1}
\\ 
&+ \frac{(\sigma_mc_{N, m} + \lambda_ms_{N, m})(\sigma_ms_{N-1, m} + \lambda_mc_{N-1, m})}{\sigma_m\cosh(\sigma_m\Delta t) + \lambda_m\sinh (\sigma_m\Delta t)} D_{N-1, m}^{\ell-1}
\\ 
=\; &-\frac{ (\sigma_m^2 - \lambda_m^2)(c_{N-1, m}s_{N, m} - s_{N-1, m}c_{N, m}) }
{ \sigma_m\cosh(\sigma_m\Delta t) + \lambda_m\sinh (\sigma_m\Delta t) } D_{N-1, m}^{\ell-1} = C_{1, m}(\Delta t) D_{N-1, m}^{\ell-1}.
\end{aligned}
\end{equation}

With the help of $R_{n, m}^\ell$ and $D_{n, m}^\ell$, we can write the iteration in a compact form. As illustrated in Figure~\ref{fig:decomposition}, there are $N-1$ interfaces, $t_1, \ldots, t_{N-1}$, where the time PSM exchanges information between two neighbouring time intervals at each iteration $\ell$. For each $t_n$, we have a Robin condition $R_{n+1, m}^\ell$, which by definition is $\dot{z}_{n+1, m}^\ell(t_{n}) + \lambda_m z_{n+1, m}^\ell(t_{n})$, in the time interval $(t_n, t_{n+1})$. Similarly, we have a Dirichlet condition $D_{n, m}^\ell$, which by definition is $z_{n, m}^\ell(t_n)$, in the time interval $(t_{n-1}, t_n)$. Using~\eqref{eq:Rn} and~\eqref{eq:Dn}, we find, for $n\in\{2, \ldots, N-2\}$, that
\[\begin{aligned}
R_{n+1, m}^\ell &= C_{1, m}(\Delta t) D_{n, m}^{\ell-1} + C_{2, m}(\Delta t) R_{n+2, m}^{\ell-1},\\ 
D_{n, m}^\ell &= C_{2, m}(\Delta t) D_{n-1, m}^{\ell-1} - \nu C_{1, m}(\Delta t) R_{n+1, m}^{\ell-1}.
\end{aligned}\]
Writing in the matrix-vector form gives
\[\begin{bmatrix}
R_{n+1, m}^\ell\\
D_{n, m}^\ell
\end{bmatrix}
=T_{l, m}\begin{bmatrix}
R_{n, m}^{\ell-1}\\
D_{n-1, m}^{\ell-1}
\end{bmatrix}
+
T_{d, m}\begin{bmatrix}
R_{n+1, m}^{\ell-1}\\
D_{n, m}^{\ell-1}
\end{bmatrix}
+
T_{r, m}\begin{bmatrix}
R_{n+2, m}^{\ell-1}\\
D_{n+1, m}^{\ell-1}
\end{bmatrix},
\]
with three auxiliary two-by-two block matrices,
\[T_{l, m} := \begin{bmatrix}
0 & 0\\
0 & C_{2, m}(\Delta t)
\end{bmatrix},
T_{d, m} := \begin{bmatrix}
0 & C_{1, m}(\Delta t)\\
-\nu C_{1, m}(\Delta t) & 0
\end{bmatrix},
T_{r, m} := \begin{bmatrix}
C_{2, m}(\Delta t) & 0\\
0 & 0
\end{bmatrix}.\]
Note that $C_{2, m}(\Delta t)$ only appears in the diagonal block, and $C_{1, m}(\Delta t)$ only in the off-diagonal blocks. Using next~\eqref{eq:D1} and~\eqref{eq:RN}, we have similarly for $n = 1$ and $n = N-1$ that,
\[\begin{aligned}
\begin{bmatrix}
R_{2, m}^\ell\\
D_{1, m}^\ell
\end{bmatrix}
&=T_{d, m}\begin{bmatrix}
R_{2, m}^{\ell-1}\\
D_{1, m}^{\ell-1}
\end{bmatrix}
+
T_{r, m}\begin{bmatrix}
R_{3, m}^{\ell-1}\\
D_{2, m}^{\ell-1}
\end{bmatrix},
\\
\begin{bmatrix}
R_{N, m}^\ell\\
D_{N-1, m}^\ell
\end{bmatrix}
&=T_{l, m}\begin{bmatrix}
R_{N-1, m}^{\ell-1}\\
D_{N-2, m}^{\ell-1}
\end{bmatrix}
+
T_{d, m}\begin{bmatrix}
R_{N, m}^{\ell-1}\\
D_{N-1, m}^{\ell-1}
\end{bmatrix}.
\end{aligned}\]
We define the vector $\boldsymbol{e}^\ell_{m} := (R_{2, m}^\ell, D_{1, m}^\ell, \, \ldots, \, R_{n+1, m}^\ell, D_{n, m}^\ell, \, \ldots, \, R_{N, m}^\ell, \, D_{N-1, m}^\ell)^\top \in \mathbb{R}^{2(N-1)}$ and the block Toeplitz matrix $T^{\text{PS}}_{N,m} \in \mathbb{R}^{2(N-1)\times 2(N-1)}$ given by
\[T^{\text{PS}}_{N,m} := \begin{bmatrix}
T_{d, m} & T_{r, m}\\
T_{l, m} & T_{d, m} & T_{r, m}\\
& \ddots & \ddots & \ddots\\
& & T_{l, m} & T_{d, m} & T_{r, m}\\
& & & T_{l, m} & T_{d, m}
\end{bmatrix}.\]
We can then write the iteration as $\boldsymbol{e}^\ell_{m} = T^{\text{PS}}_{N, m}\boldsymbol{e}^{\ell-1}_{m}$, which describes how the frequency $m$ evolves along the iterations. Note that the matrix $T^{\text{PS}}_{N,m}$ depends on the penalization parameter $\nu$ and on the time interval length $\Delta t$. We omit these dependencies in the next section for brevity.

\section{Scalability analysis}\label{sec:3}
The goal of this section is to prove that the time PSM applied to the optimality system~\eqref{eq:reduced} is weakly scalable, that is, its convergence rate does not depend on the number of fixed-size time intervals $N$. More precisely, we show that the spectral radius of $T^{\text{PS}}_{N,m}$ denoted by $\rho(T^{\text{PS}}_{N,m})$, which characterizes the \textit{asymptotic} convergence rate of the iterative method, see~\cite[Chapter 2.3]{ciaramella2022iterative}, is uniformly bounded below one with respect to $N$ for every $m\in\{1,\dots,M\}$.

We achieve our goal using two different techniques. The first approach, presented in Section \ref{sec:matrixnorm}, relies on the property that for any matrix norm $\|\cdot\|$,
\begin{equation}\label{eq:property_matrixnorm}
\rho(T^{\text{PS}}_{N,m})\leq \|T^{\text{PS}}_{N,m}\|.
\end{equation}
We therefore explicitly construct a special matrix norm $|||\cdot|||$ which verifies $|||T^{\text{PS}}_{N, m}|||\leq C<1$, with a constant $C$ \textit{independent} on $N$. 

The second approach, discussed in Section \ref{sec:Toeplitz}, relies instead on the theory of block Toeplitz matrices and allows us to derive two complementary results.
The first one is \textit{nonasymptotic}: we identify a region in the complex plane $\mathcal{D}$ that contains all eigenvalues of $T^{\text{PS}}_{N,m}$ for every $N\geq 2$. In particular, the modulus of elements in $\mathcal{D}$ is bounded by the same expression derived for $|||T^{\text{PS}}_{N, m}|||$ in Section \ref{sec:matrixnorm}, thereby recovering the same conclusion obtained there. The second result is a characterization of the \textit{asymptotic} spectrum of $T^{\text{PS}}_{N,m}$ as $N\rightarrow \infty$. Note that these two results are complementary to each other. The nonasymptotic result guarantees that the time PSM converges and is weakly scalable for every $N$. The asymptotic one provides information on the spectrum distribution for large $N$, which may provide further insights into the convergence behaviour of the iterative scheme.

To improve readability, in what follows we omit to explicit the dependence of $T^{\text{PS}}_{N,m}$ and of its submatrices on the $m-$th eigenvalue, and use instead the shorthand notations $T^{\text{PS}}_N$, $T_l$, $T_d$, and $T_r$.

\subsection{Analysis with special matrix norm}\label{sec:matrixnorm}
Since we have a closed form of each block matrix in $T^{\text{PS}}_N$, a first attempt in verifying~\eqref{eq:property_matrixnorm} consists in computing the infinity norm\footnote{Due to the particular structure of  $T^{\text{PS}}_N$, the same calculations remain valid for the $1$-norm.} of $T^{\text{PS}}_N$. A direct calculation leads to
\[\begin{aligned}
& \|T^{\text{PS}}_N\|_{\infty} = \max\big\{\|T_d + T_r\|_{\infty}, \ \|T_l + T_d + T_r\|_{\infty}, \ \|T_l + T_d\|_{\infty} \big\} 
\\
&= \max\big\{\nu |C_{1, m}(\Delta t)|, \ |C_{1, m}(\Delta t)|, \ |C_{1, m}(\Delta t)| + |C_{2, m}(\Delta t)|, \ |C_{2, m}(\Delta t)| + \nu |C_{1, m}(\Delta t)| \big\} 
\\
&= \max\big\{|C_{1, m}(\Delta t)| + |C_{2, m}(\Delta t)|, \ |C_{2, m}(\Delta t)| + \nu |C_{1, m}(\Delta t)|\big\}
\\
&= \left\{
\begin{aligned}
&|C_{1, m}(\Delta t)| + |C_{2, m}(\Delta t)|,& \text{if }& 0 < \nu \leq 1,
\\
\nu&|C_{1, m}(\Delta t)| + |C_{2, m}(\Delta t)|,& \text{if }& \nu > 1,
\end{aligned}
\right.
\\
&= \left\{
\begin{aligned}
&\frac{\sigma_m + \nu^{-1}\sinh(\sigma_m \Delta t)}{\sigma_m\cosh(\sigma_m \Delta t) + \lambda_m\sinh (\sigma_m \Delta t)},& \text{if }& 0 < \nu \leq 1,
\\
&\frac{\sigma_m + \sinh(\sigma_m \Delta t)}{\sigma_m\cosh(\sigma_m \Delta t) + \lambda_m\sinh (\sigma_m \Delta t)},& \text{if }& \nu > 1,
\end{aligned}
\right.
\end{aligned}\]
where $\lambda_m, \sigma_m, \Delta t$ are all positive. Although the infinity norm of the iteration matrix is independent of the number of time intervals $N$, it is not always smaller than one, especially when the penalization parameter $\nu$ is very small, see e.g., Figure~\ref{fig:normbounds}. Hence, the infinity norm of the iteration matrix $T^{\text{PS}}_N$ is not suitable for our purpose.

To present an alternative to estimate the spectral radius $\rho(T^{\text{PS}}_N)$ using matrix norm, we first introduce an invertible and positive definite block diagonal matrix $D\in \mathbb{R}^{2(N-1)\times 2(N-1)}$,
\[D := \begin{bmatrix}
D_d& \\
& \ddots\\
& & D_d
\end{bmatrix},
\quad
\text{ with }
\quad
D_d:=\begin{bmatrix}
1 & 0\\
0 & \sqrt{\nu}
\end{bmatrix}.\]
We then define a novel matrix norm as
\[|||T^{\text{PS}}_N||| := \max\limits_{i=1,\dots,2(N-1)}\left( \sum_{j=1}^{2(N-1)} \left(D^{-1} T^{\text{PS}}_ND\right)_{ij}^2\right)^{\frac{1}{2}},\]
which consists of a modified infinity norm applied to a similarity transformation of $ T^{\text{PS}}_N$ through the matrix $D$.

\begin{theorem}\label{thm:matrixanalysis}
The time parallel Schwarz method~\eqref{eq:PSQ2-QN-1}-\eqref{eq:PSQN} is weakly scalable in the sense that for any $\nu>0$, $\Delta t>0$ and $\lambda_m>0$, there exists a constant $C>0$, independent on $N$, such that  
\[\rho(T^{\text{PS}}_N )\leq C < 1,\]
for any $N$.
\end{theorem}

\begin{proof}
To compute explicitly $D^{-1}T^{\text{PS}}_N D$, it is convenient to manipulate separately the diagonal $ T^{\text{PS}}_{N,\text{diag}}$ and off-diagonal $ T^{\text{PS}}_{N,\text{off}}$ parts of $ T^{\text{PS}}_N$. For the diagonal part, we have
\[D^{-1}T^{\text{PS}}_{N,\text{diag}}D = 
\begin{bmatrix}
\bar T_d& \\
& \ddots\\
& &\bar T_d 
\end{bmatrix},
\quad
\bar T_d := D_d^{-1} T_d D_d = 
\begin{bmatrix}
0& \sqrt{\nu}C_{1, m}(\Delta t)\\
-\sqrt{\nu} C_{1, m}(\Delta t)& 0
\end{bmatrix}.\]
For the off-diagonal part, we find
\[D^{-1}T^{\text{PS}}_{N,\text{off}}D = 
\begin{bmatrix}
& T_r& \\
T_l& & T_r\\
& \ddots& & \ddots\\
& & T_l& & T_r\\
& & & T_l&
\end{bmatrix},\]
since $D_d^{-1} T_l D_d = T_l$ and $D_d^{-1}T_r D_d = T_r$.
Collecting the diagonal and off-diagonal contributions, we obtain
\[D^{-1}T^{\text{PS}}_ND = 
\begin{bmatrix}
\bar T_d & T_r& \\
T_l& \bar T_d & T_r\\
& \ddots& & \ddots\\
& & T_l& \bar T_d & T_r\\
& & & T_l& \bar T_d
\end{bmatrix}.\]
A direct calculation then leads to
\begin{equation}\label{eq:definitionrhotilde}
\begin{aligned}
|||T^{\text{PS}}_N|||^{2} = &\max\left\{\nu \left(C_{1,m}(\Delta t)\right)^2, \ \nu \left(C_{1,m}(\Delta t)\right)^2+\left(C_{2,m}(\Delta t)\right)^2\right\}
\\
= &\nu \left(C_{1,m}(\Delta t)\right)^2+\left(C_{2,m}(\Delta t)\right)^2
\\
= &\frac{\sigma_m^2+\nu^{-1}\sinh^{2}(\sigma_m\Delta t)}{\left(\sigma_m\cosh(\sigma_m\Delta t)+\lambda_m\sinh(\sigma_m\Delta t)\right)^2} =:\widetilde \rho(m),
\end{aligned}
\end{equation}
where in the last step we used~\eqref{eq:expression_C1C2}. Note that $\widetilde \rho(m)$ {\em does not} depend on the number of time intervals $N$. Additionally, expanding the denominator, we have
\[\begin{aligned}
\widetilde \rho(m)&=\frac{\sigma_m^2 +\nu^{-1}\sinh^2(\sigma_m\Delta t)}{\sigma_m^2\cosh^2(\sigma_m\Delta t)+\lambda_m^2\sinh^2(\sigma_m\Delta t)+2\sigma\lambda_m\sinh(\sigma_m\Delta t)\cosh(\sigma_m\Delta t)}\\
&=\frac{\lambda_m^2+\nu^{-1}\cosh^2(\sigma_m\Delta t)}{\lambda_m^2+\nu^{-1}\cosh^2(\sigma_m\Delta t)+2\lambda_m^2\sinh^2(\sigma_m\Delta t)+\sigma_m\lambda_m\sinh^2(2(\sigma_m\Delta t))},
\end{aligned}\]
where we used the properties of the hyperbolic functions and the fact that $\sigma_m^2 = \lambda_m^2 + \nu^{-1}$ in the second equality. Furthermore, the function $\widetilde \rho$ is strictly decreasing in $\lambda_m$ and reaches its maximum over the positive real line for $\lambda_m=0$, where $\widetilde \rho$ equals one. Hence, for any positive $\lambda_m$, $\nu$ and $\Delta t$, $\widetilde \rho(m)<1$, which implies that $||| T^{\text{PS}}_N|||<1$ uniformly in $N$. This concludes the proof, and in addition, proves the convergence of the time parallel Schwarz method.
\end{proof}

\subsection{Analysis based on block Toeplitz matrix theory}\label{sec:Toeplitz}
Our second approach to analyze the scalability of the time PSM leverages the theory of block Toeplitz matrices, see, e.g.,  \cite{Bottcher1997} and \cite[Chapter 2]{Trefethen2005} for an introduction.
Our goal is now twofold. First, we wish to show through this alternative path that the spectral radius $\rho(T^{\text{PS}}_N)$ remains strictly bounded below one as $N\rightarrow \infty$. Second, we aim to characterize the asymptotic distribution of the eigenvalues of $T^{\text{PS}}_N$.

We start by considering the tridiagonal block Toeplitz matrix $T^\text{PS}_N$. To align our analysis with the literature on block Toeplitz operators, we define the block entries $A_0:= T_d$, $A_{-1}:= T_r$, and $A_1:= T_l$. 

We associated with $T^\text{PS}_N$ the Laurent operator $T$ acting on the space of sequences 
\[\ell^2(\mathbb{Z},\mathbb{C}^2):=\left\{\mathbf{x}=(\mathbf{x}_j)_{j\in\mathbb{Z}},\;\mathbf{x}_j\in \mathbb{C}^2,\; \sum_{j\in\mathbb{Z}} \|\mathbf{x}_j\|^2<\infty \right\},\] and defined by the bi-infinite block-matrix
\[
T :=
\begin{pmatrix}
\ddots & \ddots &    \ddots    &        &        \\
       & A_1    & A_0    & A_{-1} &        \\
       & & A_1    & A_ 0   & A_{-1} &  \\
       &        &   & \ddots  & \ddots &  \ddots
\end{pmatrix}.
\]
Let $(A_j)_{j\in\mathbb{Z}}$ denote the sequence of matrix coefficients where, in our case, $A_j\equiv 0$ if $|j|>1$. The action of $T$ onto a sequence $\mathbf{x}$ can be expressed as the discrete convolution $(T\mathbf{x})_i=\sum_{j\in\mathbb{Z}} A_{i-j}\mathbf{x}_j$. By applying the time discrete Fourier transform, $\mathcal{F}(\mathbf{x})=\sum_{j\in\mathbb{Z}} \mathbf{x}_je^{-ij\theta}=:\widehat{\mathbf{x}}(\theta)$, one obtains the fundamental relation $\mathcal{F}\left(T \mathbf{x}\right)(\theta)=F(\theta)\widehat{\mathbf{x}}(\theta)$, where
$F:[-\pi,\pi]\rightarrow \mathbb{C}^{2\times 2}$ is the matrix-valued symbol of $T$,
\begin{equation}\label{eq:Ftheta}
F(\theta):=A_{-1} e^{-i\theta} + A_0  +A_1 e^{i\theta}=
\begin{pmatrix}
C_{2,m}(\Delta t)e^{-i\theta} & \sqrt{\nu} C_{1,m}(\Delta t)\\
-\sqrt{\nu} C_{1,m}(\Delta t) & C_{2,m}(\Delta t)e^{i\theta}
\end{pmatrix},
\end{equation}
for any $\theta\in [-\pi,\pi]$. The symbol $F$ generates the operator $T$ in the sense that the coefficients $A_j$ are recovered by
\[A_j=\frac{1}{2\pi}\int_{-\pi}^\pi F(\theta)e^{-i j\theta}d\theta,\]
and our matrix $T^{\text{PS}}_N$ coincides with the finite Toeplitz matrix $T_{n}=(A_{j-i})_{i,j=1}^{n}$ for $n=N-1$, where $A_j\equiv 0$ if $|j|>1$ in our case.

The spectral properties of Laurent operators are well-established. Specifically, the spectrum\footnote{For an infinite dimensional operator $T$, the spectrum is the set of values $\lambda\in \mathbb{C}$ such that $(T-\lambda I)^{-1}$ is not boundedly invertible.} of $T$, denoted by $\sigma(T)$, coincides with the union of the spectra of $F(\theta)$ for $\theta \in [-\pi,\pi]$,
\begin{equation}\label{eq:spectrumT}
\sigma(T):=\cup_{\theta\in [-\pi,\pi]} \, \sigma\left(F(\theta)\right).
\end{equation}
A direct calculation reveals that the two eigenvalues of $F(\theta)$ are 
\begin{equation}\label{eq:spetrumFtheta}
\mu_{\pm}(\theta)=C_{2,m}(\Delta t)\cos(\theta)\pm i \sqrt{\left(C_{2,m}(\Delta t)\sin(\theta)\right)^2+\nu (C_{1,m}(\Delta t))^2},
\end{equation}
so that $\sigma(T)$ consists of two, distinct, closed curves in the complex plane
\begin{equation}\label{eq:explicitdescriptionSigmaT}
\sigma(T)=\{z\in\mathbb{C}:\; z= \mu_{+}(\theta),\;\theta\in [-\pi,\pi]\}\cup \{z\in\mathbb{C}:\; z= \mu_{-}(\theta),\;\theta\in [-\pi,\pi]\}.
\end{equation}

The relationship between the eigenvalues of the finite Toeplitz matrix $T_n$, denoted by $\{\widetilde \lambda_j\}_{j=1}^{2n}$, and those of its Laurent operator has been the subject of extensive research. To provide a brief overview, we recall the definition of an eigenvalue cluster \cite{schmidt1960toeplitz,tilli1999some,tyrtyshnikov1996unifying}.
\begin{definition}[Eigenvalue cluster]\label{def:eigenvaluecluster}
A subset $\Phi\subset \mathbb{C}$ is called an \textit{eigenvalue cluster} for the sequence $\{T_n\}_{n\in\mathbb{N}}$ if for any open set $\mathcal{O}$, $\Phi\subset \mathcal{O}$, the number of eigenvalues of $T_n$ that lie outside $\mathcal{O}$ is a $o(n)$ as $n\rightarrow \infty$.
\end{definition}
A landmark theorem by Szeg\H{o} \cite[Chapter 5.2]{grenander2001toeplitz} states that for a Hermitian Toeplitz matrix generated by a bounded real-valued symbol $f:[-\pi,\pi]\rightarrow \mathbb{R}$, i.e. $A_j\in\mathbb{C}$ and $A_j=A^{\star}_{j}$, with eigenvalues $\{\widetilde \lambda_j\}_{j=1}^n$, the relation
\begin{equation}\label{eq:Szego}
\lim_{n\rightarrow \infty} \frac{1}{n}\sum_{j=1}^n g(\widetilde \lambda_j)=\frac{1}{2\pi}\int_{-\pi}^\pi g(f(\theta))d\theta,
\end{equation}
holds for any continuous function $g$.
By choosing suitable test functions $g$, the relation~\eqref{eq:Szego} implies that $\sigma(T)$ is an eigenvalue cluster for $\{T_n\}_{n\in\mathbb{N}}$.
Szeg\H{o}'s result has been extended to Hermitian Toeplitz matrices with matrix-valued symbols \cite{miranda2000asymptotic} and to non-Hermitian with special structure (e.g., banded) \cite{hirschman1967spectra, schmidt1960toeplitz,MR219986}. For general non-Hermitian matrices, with possibly matrix-valued symbols, it is known that \eqref{eq:Szego} does in general not hold, unless one restricts $g$ to the class of holomorphic functions over a suitable region of the complex plane or additional assumptions are imposed on the range of the symbol $F$ \cite{tilli1998singular,tilli1999some}.

To start our analysis, we recall the notion of numerical range\footnote{As the symbol and its eigenvalues studied in this section depend continuously on $\theta$, we used a simplified notation compared to the works devoted to general measurable symbols, see, e.g., \cite{tilli1998singular}.}. For a square matrix $A\in\mathbb{C}^{k\times k}$, the numerical range is defined as
\begin{equation}\label{eq:numericalrangeMatrix}
\R(A):=\left\{\frac{\langle A \boldsymbol x, \boldsymbol x\rangle}{\|\boldsymbol x\|^2}:\; \boldsymbol x\in\mathbb{C}^k,\; \boldsymbol x\neq 0\right\}.
\end{equation}
In the scalar case, where $f:[-\pi,\pi]\rightarrow \mathbb{C}$, the numerical range is set equal to the image of the function,
\[\R(f):=\left\{z\in\mathbb{C}:z=f(\theta),\;\theta\in [-\pi,\pi]\right\}.\]
For a matrix-valued function $F:[-\pi,\pi]\rightarrow \mathbb{C}^{k\times k}$, its numerical range is defined as
\[\R(F):=\cup_{\theta\in [-\pi,\pi]} \R(F(\theta)).\]
For our analysis, we need the additional concept of
\textit{extended} numerical range of a matrix-valued function $F$ \cite[Definition 5.1]{tilli1998singular}, defined as
\[\mathcal{E}\R(F):=\cap_{\mathcal{H}\in\Theta} \mathcal{H},\]
where $\Theta$ is the family of all closed half planes $\mathcal{H}\subset \mathbb{C}$ such that $\R(F(\theta))\subset \mathcal{H}$ holds for every $\theta\in[-\pi,\pi]$.

\begin{theorem}
For any $n\in\mathbb{N}$, the eigenvalues of $T_n$ are included in the region of complex plane,
\[\begin{aligned}
\mathcal{D}:=\Bigl\{z\in\mathbb{C}:\;\exists \theta\in [-\pi,\pi]\text{ s.t. } &\Re z = C_{2,m}(\Delta t)\cos(\theta),\\
& |\Im z| \leq \sqrt{(C_{2,m}(\Delta t)\sin(\theta))^2+\nu(C_{1,m}(\Delta t))^2 }\Bigr\}.
\end{aligned}\]
In particular, the spectral radius of $ T^{\text{PS}}_{N}$ is bounded by $\widetilde \rho(m)$ below one, uniformly with respect to $N$. 
\end{theorem}
\begin{proof}
The proof relies on \cite[Theorem 5.1]{tilli1998singular}, which states that, for any $n\in\mathbb{N}$, the eigenvalues of $T_n$ belongs to $\mathcal{E}\R(F)$. We are thus left to show that $\mathcal{E}\R(F)=\mathcal{D}$.
To this end, we first observe that, for a fixed $\theta\in [-\pi,\pi]$, $F(\theta)$ defined in~\eqref{eq:Ftheta} is a normal matrix, thus $\R(F(\theta))$ is the convex hull of its eigenvalues $\mu_{\pm}(\theta)$, representing geometrically a vertical segment in the complex plane. Then, $\mathcal{D}=\cup_{\theta\in [-\pi,\pi]}\R(F(\theta))$ is the region in the complex plane obtained as the union of the vertical segments centered at $C_{2,m}(\Delta t) \cos(\theta)$, $\theta\in [-\pi,\pi]$, with half-height $\sqrt{(C_{2,m}(\Delta t)\sin(\theta))^2+\nu(C_{1,m}(\Delta t))^2}$. In particular, $\mathcal{D}$ is a closed and convex set.
Take now a $z\in \mathcal{D}$. Then, $z$ belongs to $\R(F(\theta))$ for a certain value of $\theta$ and thus, by the definition of $\mathcal{E}\mathcal{R}(F)$, $z$ belongs to every $\mathcal{H}\in\Theta$. Hence, $z$ also belongs to the intersection, which implies that $z\in \mathcal{E}\mathcal{R}(F)$ and $\mathcal{D}\subset \mathcal{E}\mathcal{R}(F)$. To show that $\mathcal{E}\mathcal{R}(F)\subset \mathcal{D}$, we take a $z\notin \mathcal{D}$. Since $\mathcal{D}$ is closed and convex, there exists a closed half plane $\widetilde{\mathcal{H}}$ such that $\mathcal{D}\subset \widetilde{\mathcal{H}}$ and $z\notin \widetilde{\mathcal{H}}$, thus $z\notin \mathcal{E}\mathcal{R}(F)$. Combining these two parts, we conclude that $\mathcal{E}\mathcal{R}(F)=\mathcal{D}$.
Finally, we observe that for any $z\in \mathcal{D}$, $\|z\|$ is bounded by $\widetilde \rho(m)$ defined in \eqref{eq:definitionrhotilde}, thus proving the second claim.
\end{proof}

To prove an asymptotic result on the distribution of the eigenvalues of $\{T_n\}_{n\in\mathbb{N}}$, we recall \cite[Theorem 1.2]{donatelli2012canonical} which proves a Szeg\H{o}-type limit for a sequence of (multilevel) block Toeplitz matrices with non-Hermitian symbols. Here, the statement is reported in our simplified setting.
\begin{theorem}[Theorem 1.2, Ref. \cite{donatelli2012canonical}]\label{eq:theorem_canonical}
Let $F\in L^{\infty}([-\pi,\pi];\mathbb{C}^{k\times k})$ be a matrix-valued symbol with eigenvalues $\left\{\mu_j\right\}_{j=1}^k$. Furthermore, let $B_n\in \mathbb{C}^{nk\times nk}$ be the Toeplitz matrix generated by $F$, denote its eigenvalues by $\{\widetilde \lambda_j\}_{j=1}^{nk}$, and let $B$ be the associated Laurent operator. If $\sigma(B)$ has empty interior and does not disconnect the complex plane, then for every continuous function $g:\mathbb{C}\rightarrow \mathbb{C}$ with bounded support,
\[\lim_{n\rightarrow \infty}\frac{1}{nk}\sum_{j=1}^{nk} g(\widetilde \lambda_j)=\frac{1}{2\pi}\int_{-\pi}^\pi \frac{1}{k} \sum_{j=1}^k g(\mu_j)d\theta.\]
\end{theorem}

As discussed in \cite[Section 2]{donatelli2012canonical}, Theorem \ref{eq:theorem_canonical} implies that $\sigma(B)$ is a cluster for $\{B_n\}_{n\in\mathbb{N}}$ in the sense of Definition~\ref{def:eigenvaluecluster}.
In the next result, we show that, in our setting, $\sigma(T)$ satisfies the assumptions of Theorem \ref{eq:theorem_canonical}, thus concluding that $\sigma(T)$ is a cluster for the sequence $\{T_n\}_{n\in\mathbb{N}}$. We will verify in Section~\ref{sec:spectrum} that the eigenvalues of $T_N^{\text{PS}}$ accumulate on $\sigma(T)$ as $N$ increases.
\begin{corollary}\label{cor:clustering}
$\sigma(T)$ is a cluster for the sequence $\{T_n\}_{n\in\mathbb{N}}$ in the sense of Definition~\ref{def:eigenvaluecluster}.
\end{corollary}
\begin{proof}
From \eqref{eq:spectrumT} and \eqref{eq:spetrumFtheta}, it follows that $\sigma(T)$ consists of two, distinct, closed curves in the complex plane defined by the maps $\theta\mapsto \mu_-(\theta)$ and $\theta\mapsto \mu_{+}(\theta)$. Thus, $\sigma(T)$ has empty interior. In addition, $\sigma(T)$ disconnects the complex plane if and only if $C_{1,m}(\Delta t)=0$. This requires that $\sigma_m\Delta t=0$, which is never the case since $\nu$, $\Delta t$, and $\lambda_m$ are all positive.
Hence, we have verified the conditions of Theorem~\ref{eq:theorem_canonical} and the claim follows.
\end{proof}

\section{Numerical experiments}\label{sec:4}

The numerical experiments are divided into three groups. Section \ref{sec:spectrum} verifies numerically the bound of Theorem \ref{thm:matrixanalysis} and the eigenvalue clustering described by Corollary \ref{cor:clustering}. Section \ref{sec:numschwarz} compares the theoretical bounds on the spectral radius of the iteration matrix with the actual asymptotic convergence of the time parallel Schwarz method. Finally, Section \ref{sec:numcooling} presents an application of the numerical framework to the optimal control of a periodic heating-cooling device over a growing time window.

\subsection{Matrix bounds and spectrum clustering}\label{sec:spectrum}

Figure \ref{fig:normbounds} compares the behavior of $\rho(T^{\text{PS}}_{N,m})$, $\|T^{\text{PS}}_{N,m}\|_{\infty}$ and $\widetilde \rho(m)$ for different numbers of time intervals $N$. The left panel refers to $m=1$, that is, the smallest eigenvalue of $A$, while the right panel refers to the largest eigenvalue for $m=M$.
We first observe that the infinity norm can be much larger than $1$, and thus not suitable for our scalability analysis. In contrast, the bound $\widetilde \rho(m)$ derived in Section \ref{sec:matrixnorm} is very sharp, smaller than one, and most importantly independent of $N$. Secondly, we observe the classical smoothing behaviour of a Schwarz method: low-frequency spatial error components, that is, those associated with small eigenvalues, converge much slower than higher frequency components, and therefore determine the overall convergence behaviour of the algorithm.
This is further confirmed by Figure \ref{fig:widetilderho}, which shows the behaviour of $\widetilde \rho(m)$ with respect to $m$ for different values of $\nu$ and $\Delta t$.
In particular, we notice that as $\Delta t\rightarrow 0$ and $\nu\rightarrow 0$, $\widetilde \rho(1)\rightarrow 1$, thus the overall convergence rate of the time PSM is expected to deteriorate.

\begin{figure}
\centering
\includegraphics[width=0.48\textwidth]{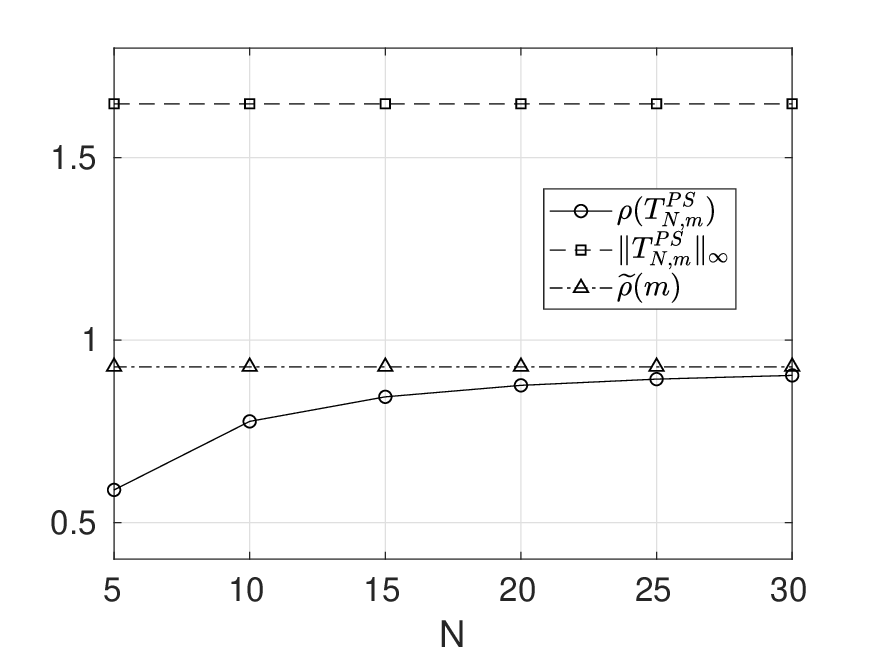}
\includegraphics[width=0.48\textwidth]{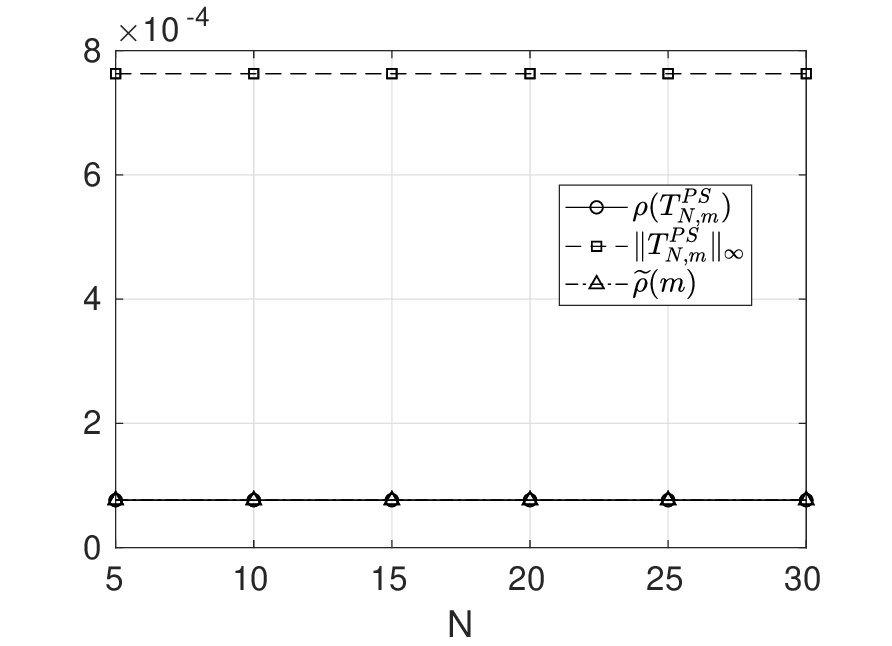}
\caption{Comparison of $\rho(T^{\text{PS}}_{N,m})$, $\|T^{\text{PS}}_{N,m}\|_{\infty}$ and $\widetilde \rho(m)$ for different numbers of time intervals $N$. Left panels refer to $m=1$ while right panels refer to $m=M$. The remaining parameters are: $M=128$, $\nu=10^{-2}$ and $\Delta t=1/M$.
}\label{fig:normbounds}
\end{figure}

\begin{figure}
\centering
\includegraphics[width=0.48\textwidth]{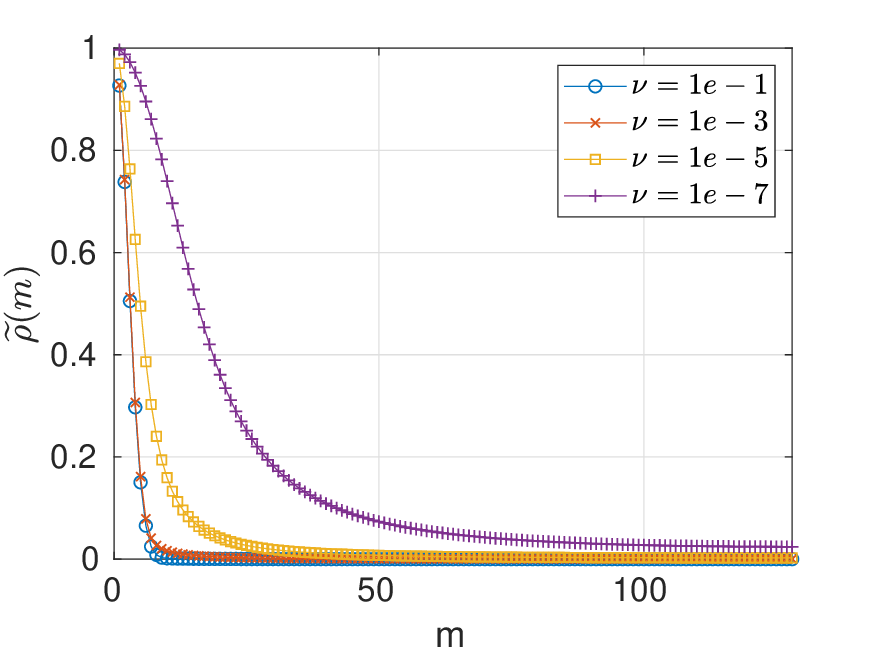}
\includegraphics[width=0.48\textwidth]{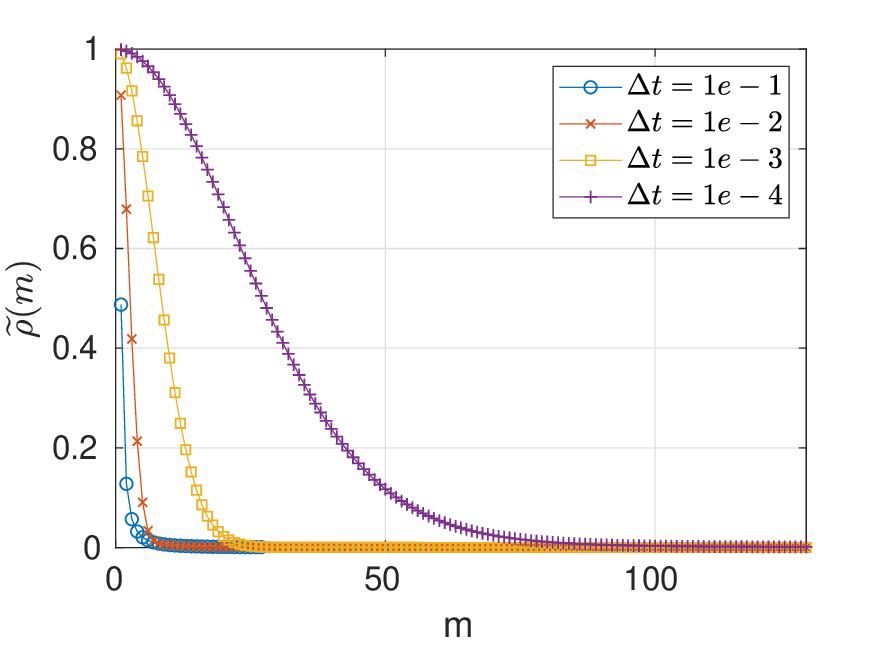}
\caption{Plot of the map $m\rightarrow \widetilde \rho(m;\nu,\Delta t)$ for different values of $\nu$ and $\Delta t$. Parameters: $M=128$, $\Delta t=1/M$ (left panel), $M=128$, $\nu=10^{-2}$ (right panel).}\label{fig:widetilderho}
\end{figure}

Next, we verify numerically the clustering of the eigenvalues of the Toeplitz matrices $\{T_{N,m}^{\text{PS}}\}_{N}$ discussed in Section \ref{sec:Toeplitz}. The first row of Figure \ref{fig:spectra} shows the set $\sigma(T)$ in the complex plane (see \eqref{eq:explicitdescriptionSigmaT}), together with the discrete eigenvalues $\{\widetilde \lambda_j\}_{j=1}^{2(N-1)}$ of $T_{N,1}^{\text{PS}}$ (i.e. $m=1$) for different values of $N$.
The second rows refers instead to $\{T_{N,M}^{\text{PS}}\}_{N}$. Notice that for large $\lambda_M$, $\sigma_M=\sqrt{\lambda_M^2+\nu^{-1}}$ is large, and consequently $C_{2,m}(\Delta t)$ defined in~\eqref{eq:expression_C1C2} tends rapidly to zero, and this explains why $\sigma(T)$ collapes into two points along the imaginary axis.
We observe that the eigenvalues accumulate in the region $\sigma(T)$. Although, Definition \ref{def:eigenvaluecluster} and Theorem \ref{cor:clustering} allow for a moderate number $P(N)$ of eigenvalues that do not cluster on $\sigma(T)$ (with $P(N)/N\rightarrow 0$ as $N\rightarrow \infty$), we have not observed any such outlier in our experiments.

\begin{figure}
\centering
\includegraphics[width=0.48\textwidth]{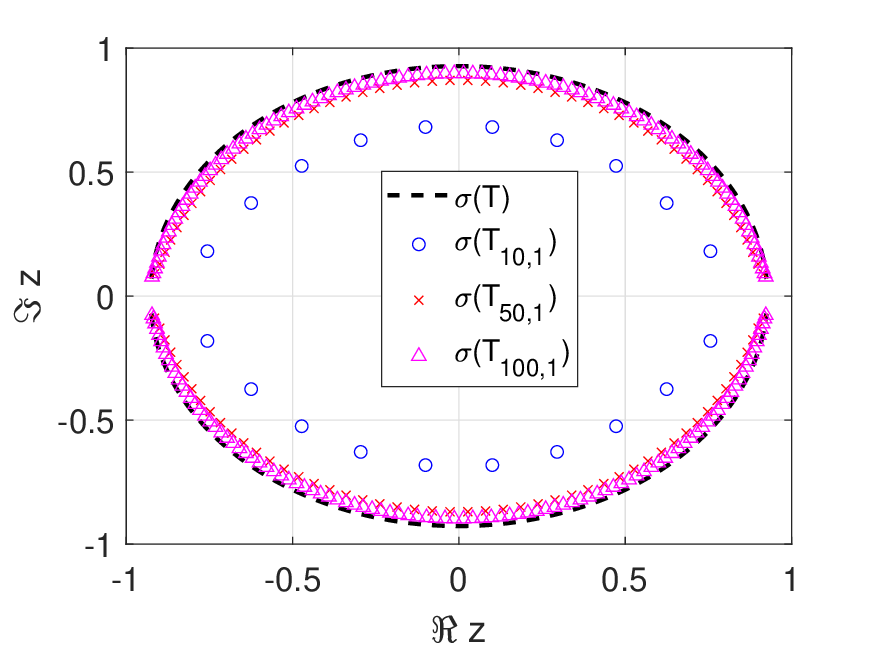}
\includegraphics[width=0.48\textwidth]{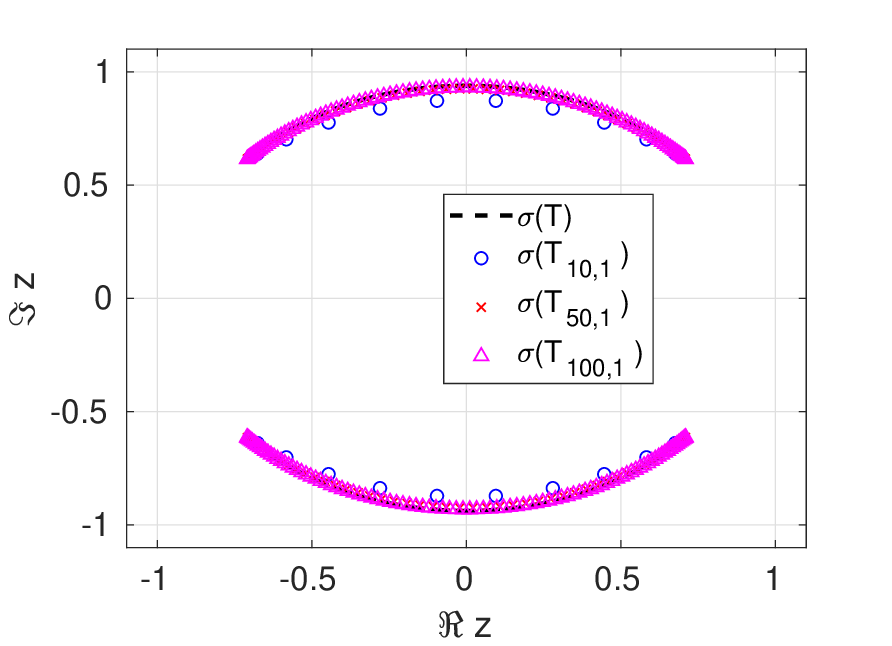}
\includegraphics[width=0.48\textwidth]{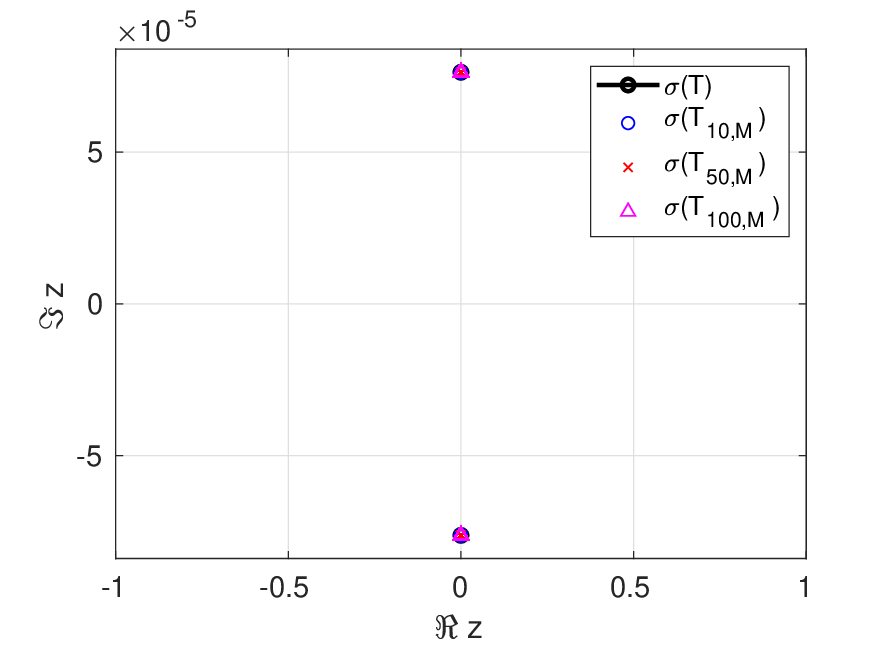}
\includegraphics[width=0.48\textwidth]{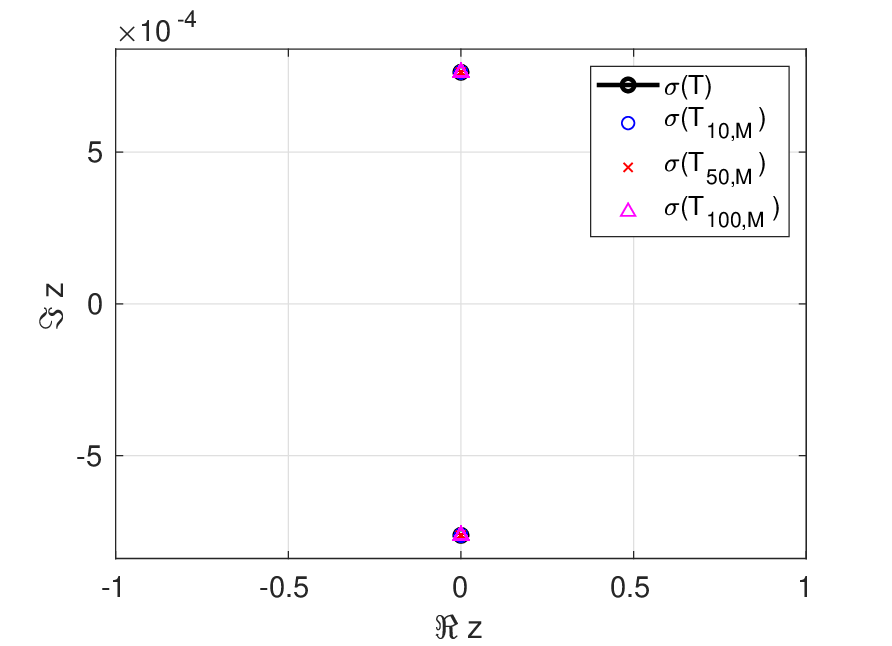}
\caption{Clustering of the eigenvalues of $\{T_{N,m}^{\text{PS}}\}_{N}$ for $m=1$ (top row) and $m=M$ (bottom row). Parameters: $\Delta t=1/M$ and $\nu=10^{-2}$ (left panels) and $\nu=10^{-4}$ (right panels). In the legend, we omit the superscript $\text{PS}$ to improve readability.}\label{fig:spectra}
\end{figure}

\subsection{Weak Scalability of time parallel Schwarz algorithm}\label{sec:numschwarz}

We now test the weak scalability of our time PSM~\eqref{eq:PSQ2-QN-1}-\eqref{eq:PSQN} numerically, and compare the convergence with the theoretical bounds on the spectral radius of the iteration matrix. We consider the manufactured solutions, 
\[\begin{aligned}
y(x, t)&=\sin(\pi x)\left(te^{-\pi^2t} - \frac{e^{-\pi^2T}}{1+\pi^2T}t\right), 
\\
p(x, t)&=\nu\sin(\pi x)\left(e^{-\pi^2t} - e^{-\pi^2T} \frac{1+\pi^2t}{1+\pi^2T}\right),
\end{aligned}\] 
which satisfy the reduced optimality system~\eqref{eq:reduced} together with the target function 
\[\hat y(x, t) = \nu\sin(\pi x)\left(\left(\frac t{\nu} + 2\pi^2\right)e^{-\pi^2t} - e^{-\pi^2T}\frac{\frac t{\nu} + \pi^2t}{1+\pi^2T}\right).\]
We then take this target function, the initial condition $y_0=0$, the space domain $\Omega=(0, 1)$ and the homogeneous Dirichlet boundary condition $g=0$ for our numerical tests. To solve our parabolic optimal control problem~\eqref{eq:J}-\eqref{eq:heat}, we follow the optimize-then-discretize approach that consists of discretizing the reduced optimality system~\eqref{eq:reduced}. We use the Crank-Nicolson method, which gives a second-order approximation of the exact solutions of both state and adjoint variables as shown in Figure~\ref{fig:CNcv}.
\begin{figure}
\centering
\includegraphics[scale = 0.4]{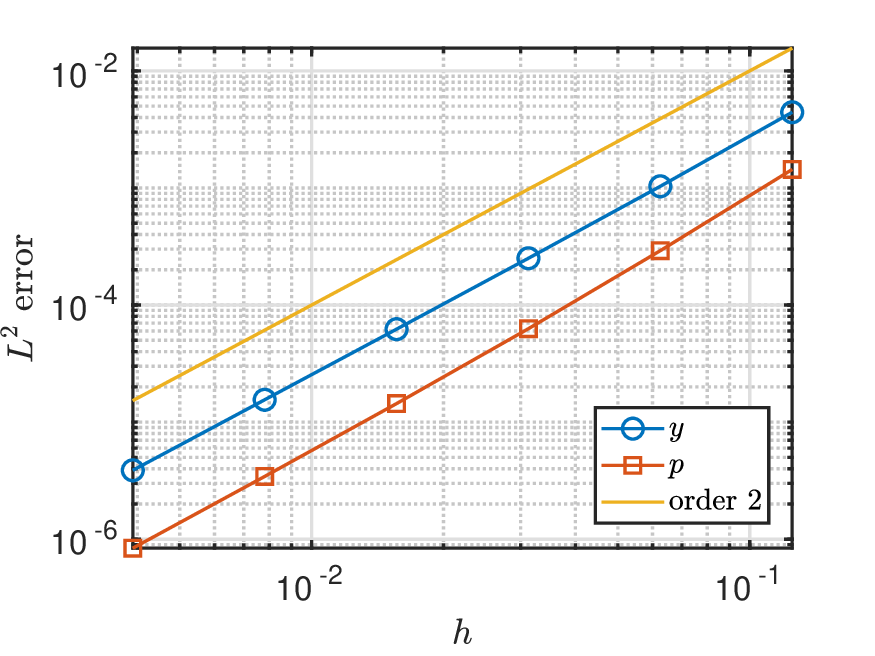}
\caption{Second order convergence of the Crank--Nicolson method applied to solve the reduced optimality system~\eqref{eq:reduced}. The mesh size and time step satisfy $h_x=h_t=h$, and $h$ varies in the set $\{2^{-8}, 2^{-7}, \ldots, 2^{-3}\}$.}
\label{fig:CNcv}
\end{figure}

We now set $\nu = 1/10$, choose the time step and mesh size as $h_t = h_x = 1/32$ and present a weak scalability test using different time interval lengths $\Delta t$. \begin{figure}
\centering
\includegraphics[scale = 0.3]{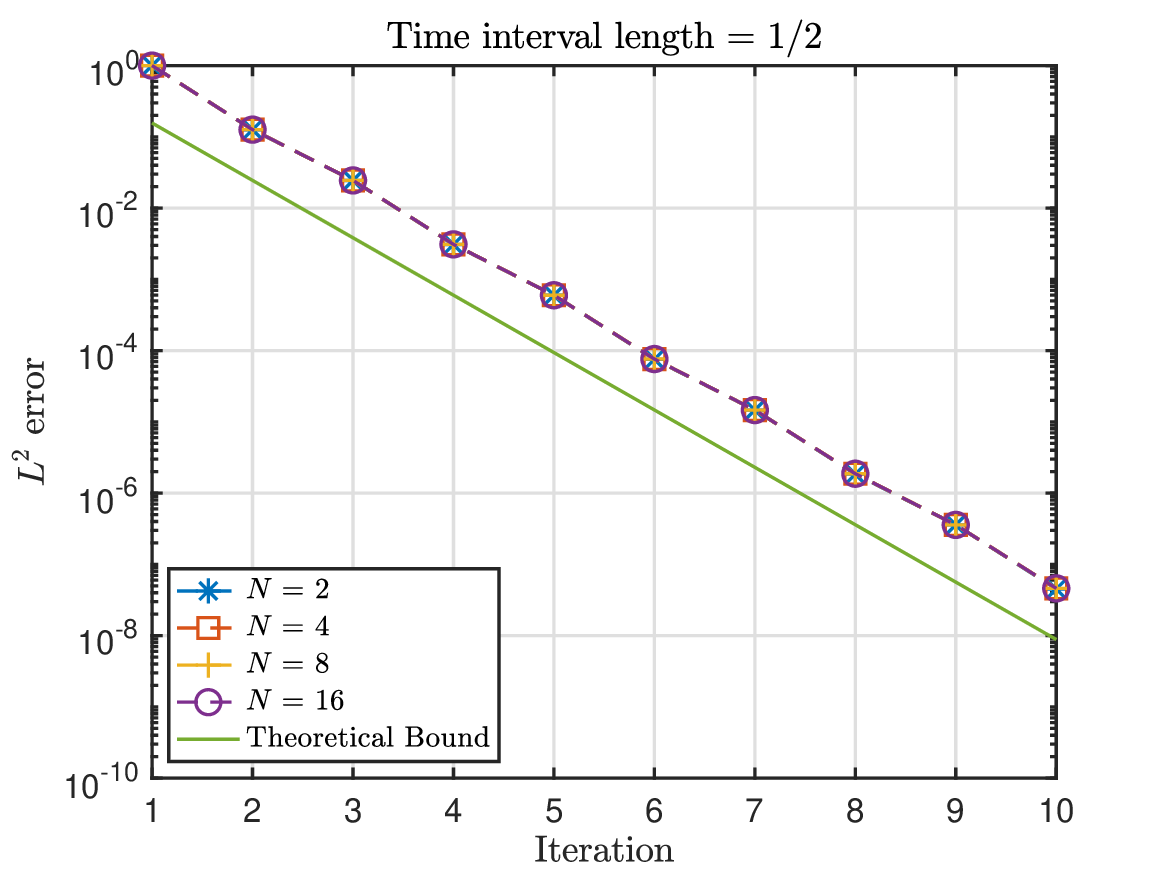}
\includegraphics[scale = 0.3]{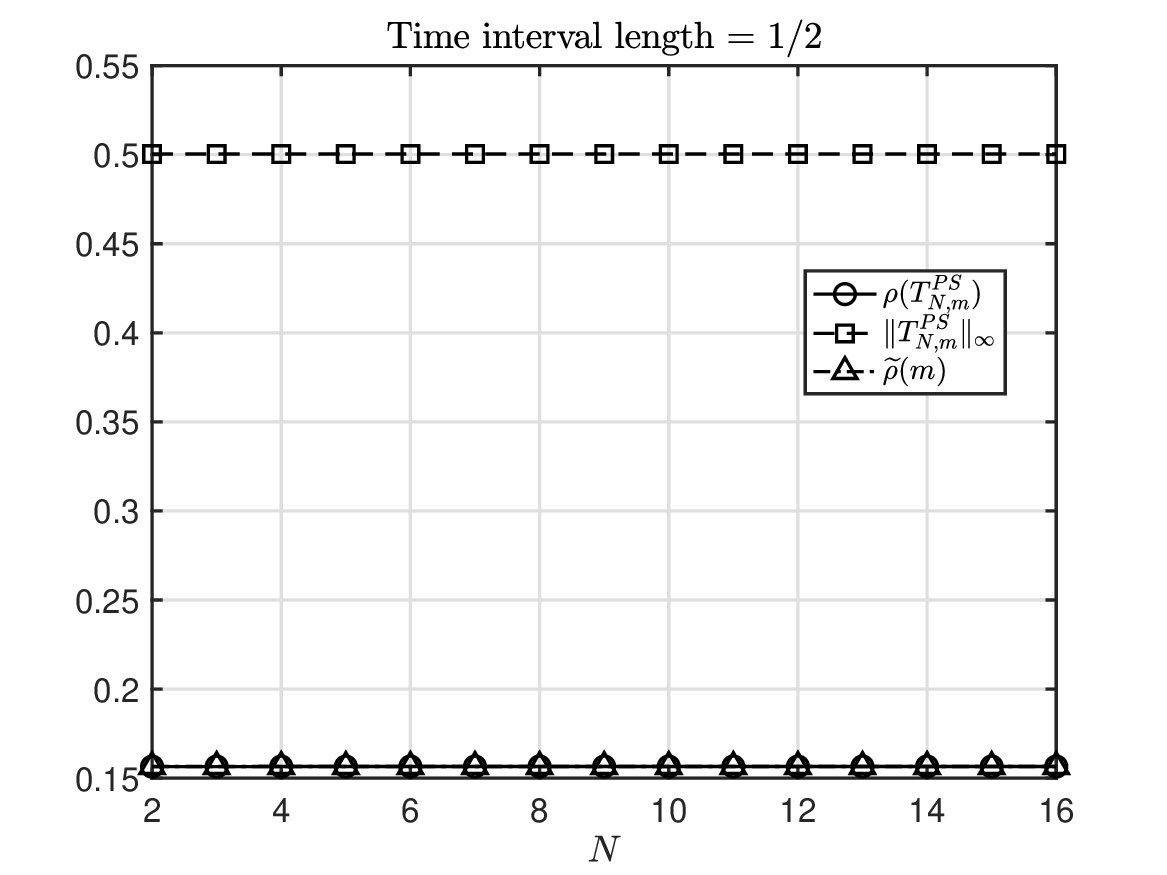}
\includegraphics[scale = 0.3]{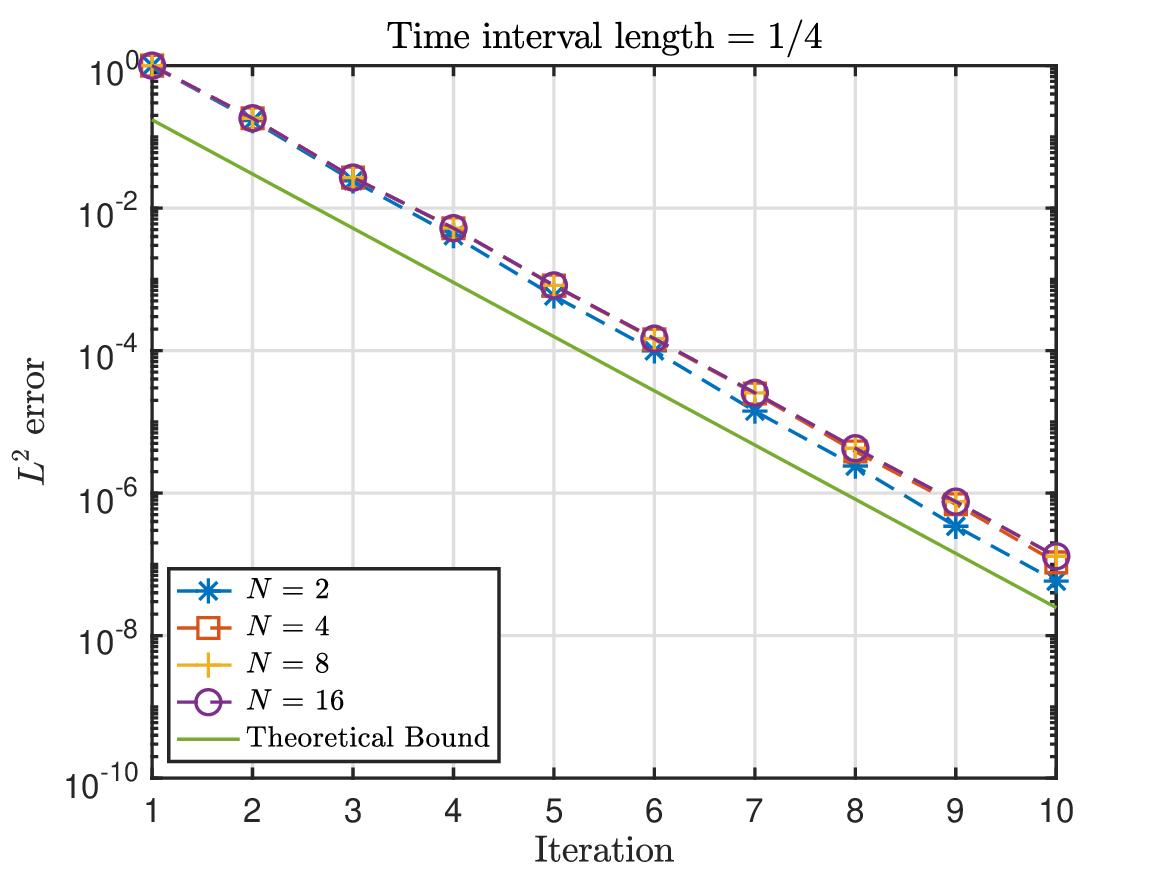}
\includegraphics[scale = 0.3]{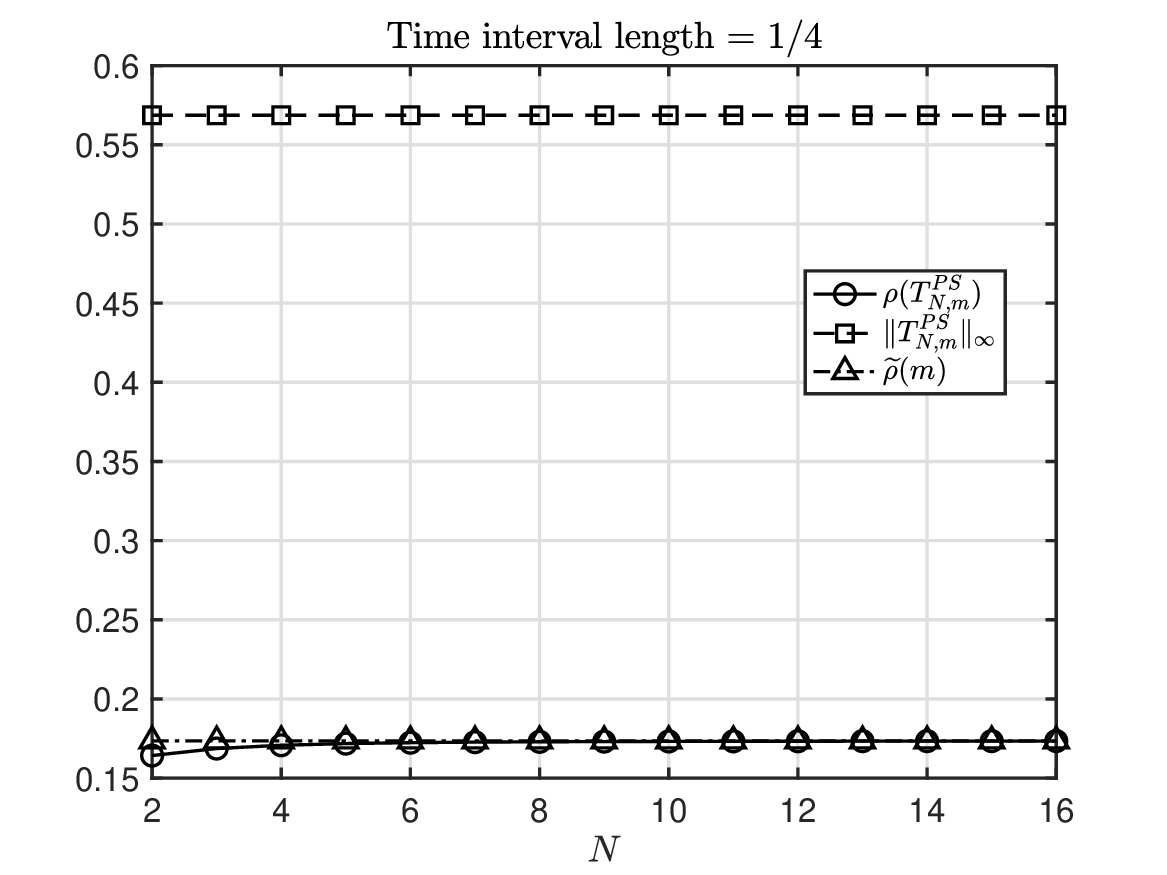}
\includegraphics[scale = 0.3]{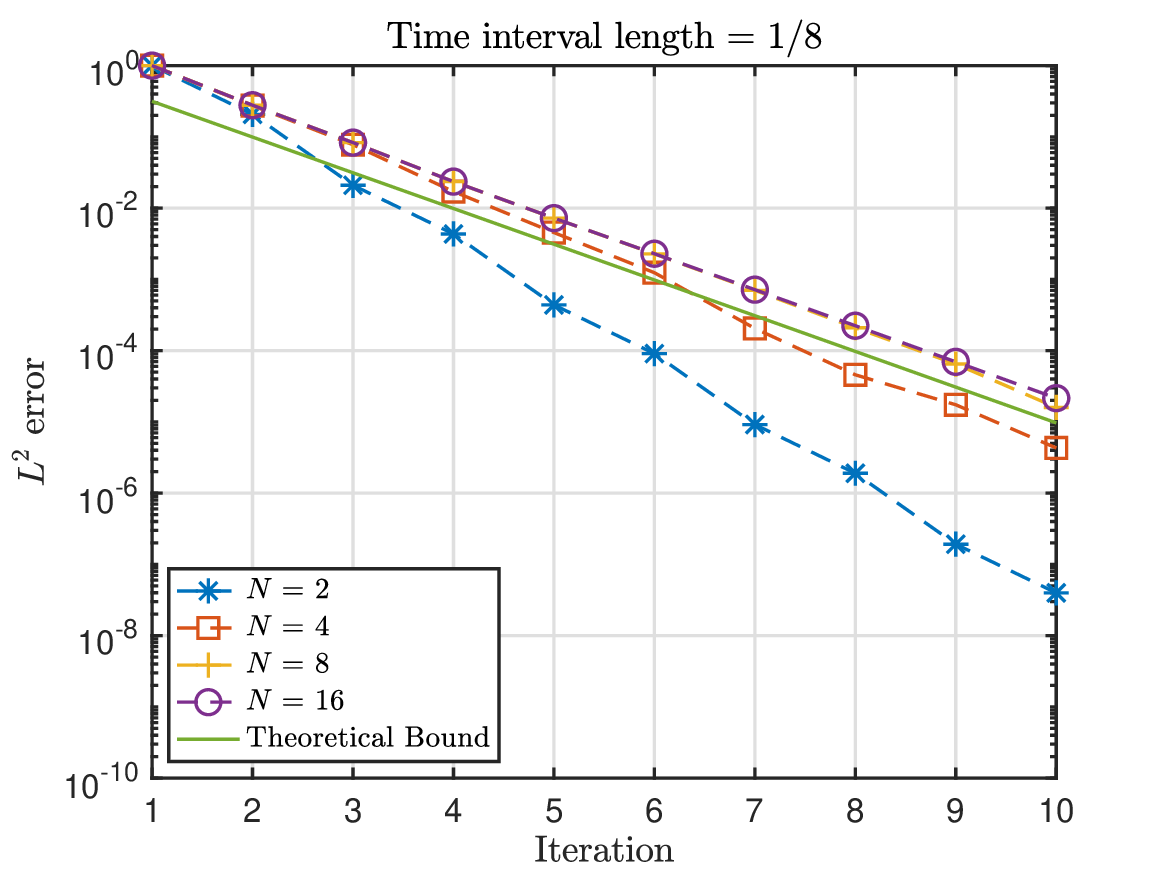}
\includegraphics[scale = 0.3]{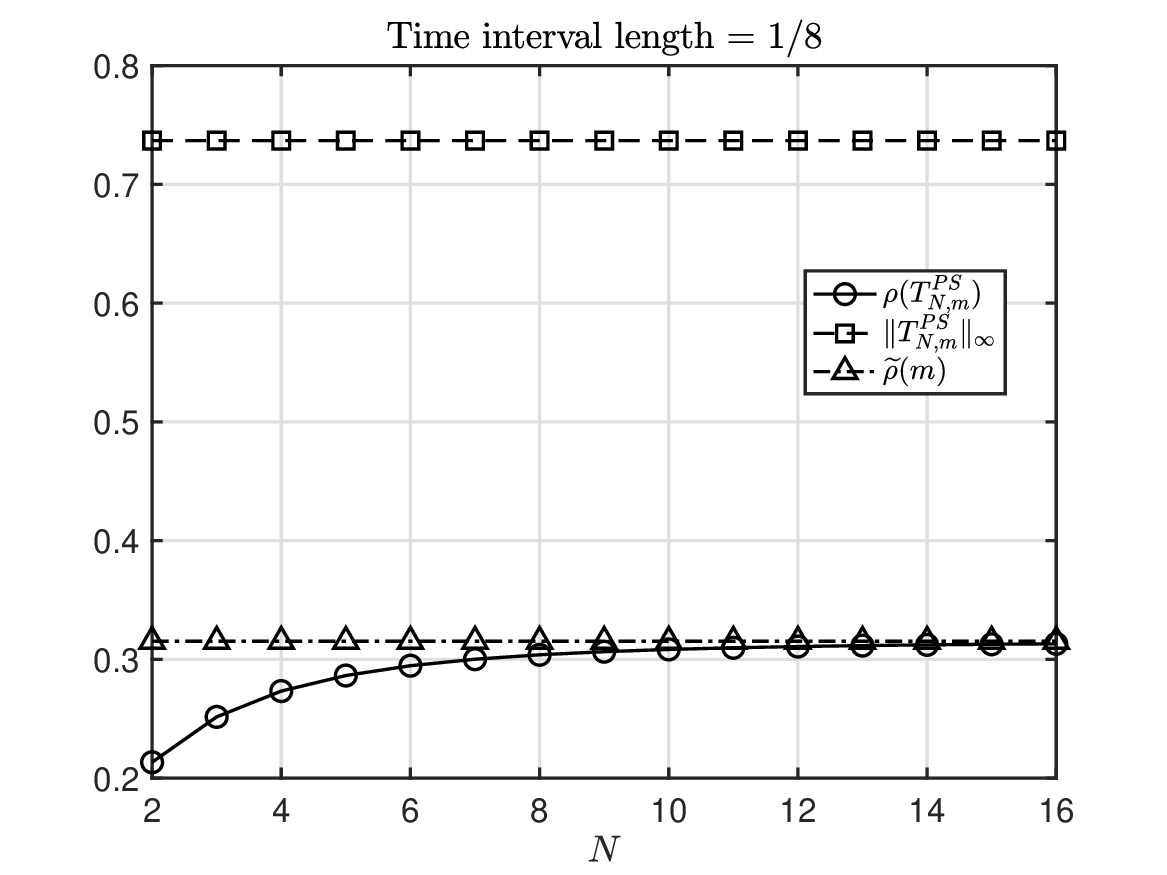}
\includegraphics[scale = 0.3]{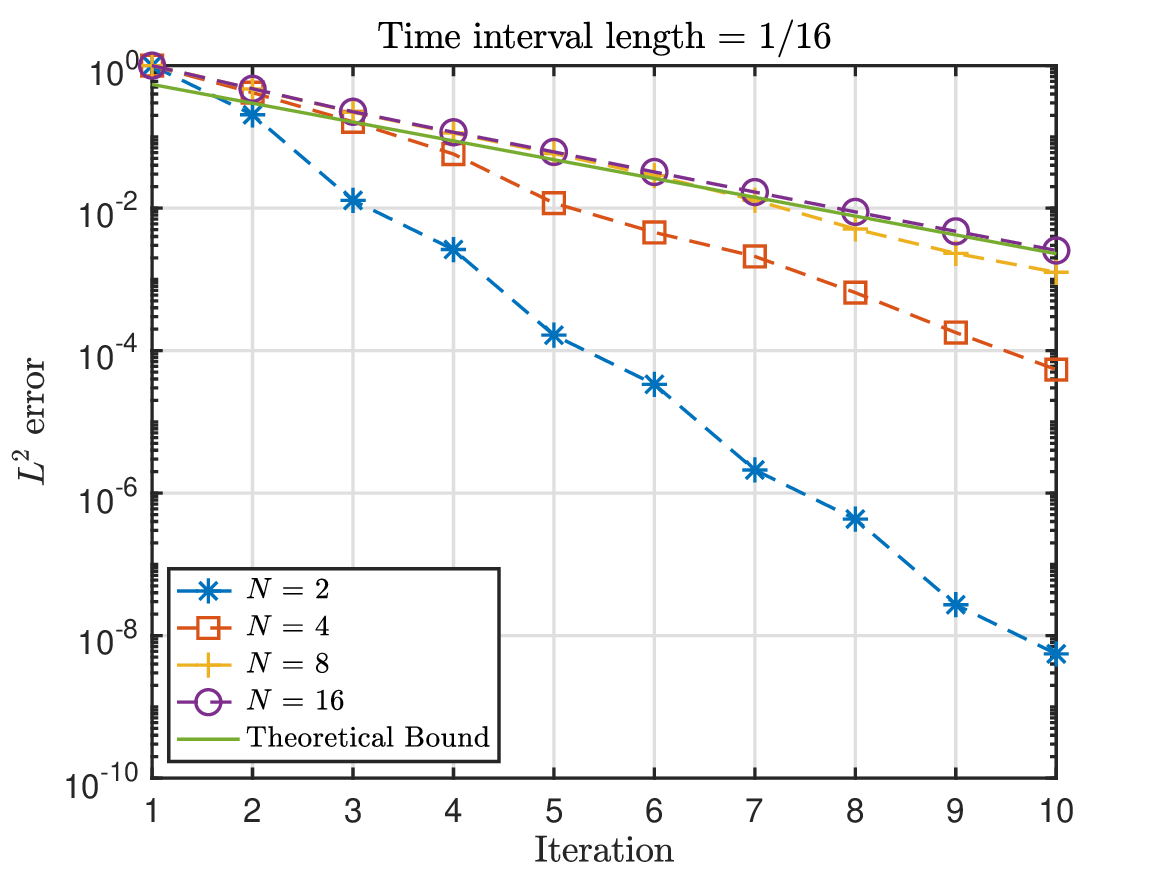}
\includegraphics[scale = 0.3]{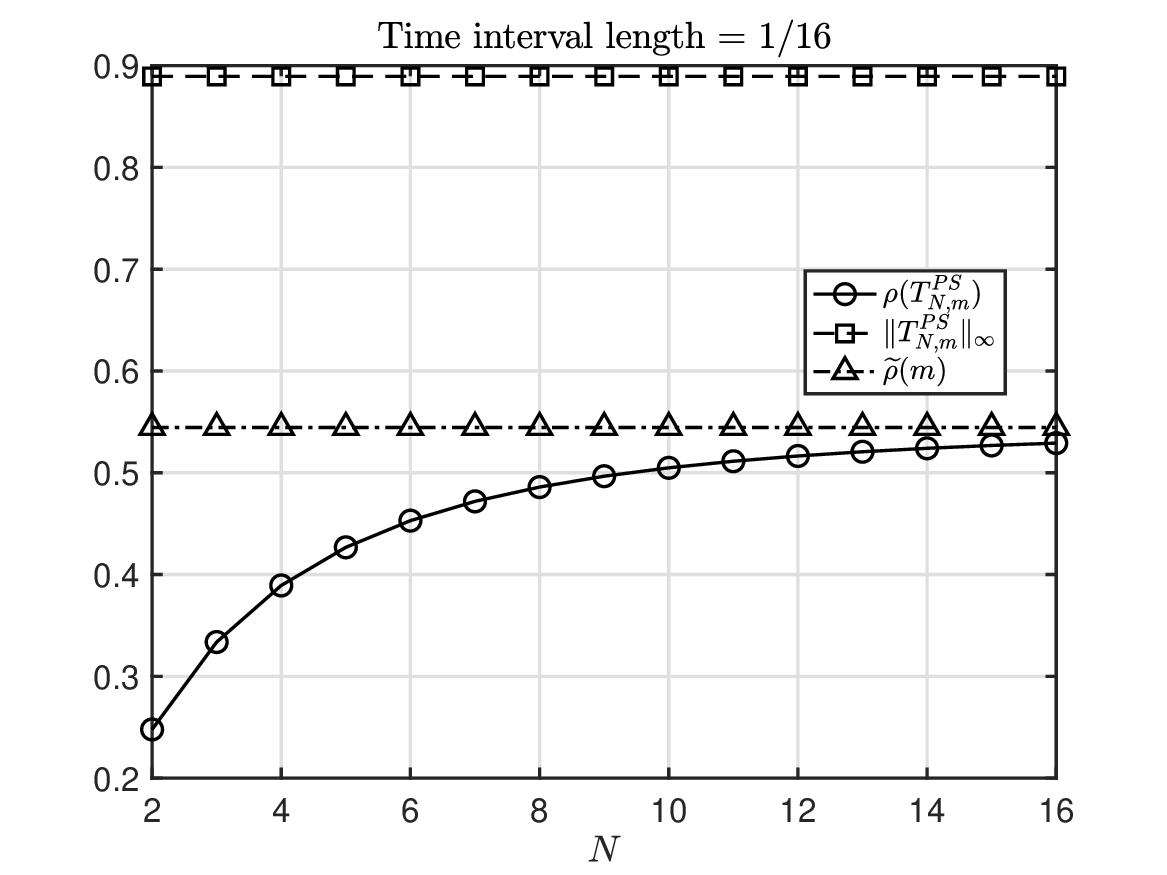}
\caption{Left column: weak scalability test of the time PSM~\eqref{eq:PSQ2-QN-1}-\eqref{eq:PSQN} for $N\in\{2, 4, 8, 16\}$. Right column: $\rho(T^{\text{PS}}_{N,m})$, $\|T^{\text{PS}}_{N,m}\|_{\infty}$ and $\widetilde \rho(m)$ with $m=1$ for different numbers of time intervals $N$. Parameters: $h_t=h_x=1/32$, $\nu=10^{-1}$.}
\label{fig:WSTest}
\end{figure}
Specifically, the left column of Figure~\ref{fig:WSTest} shows the decay of the $L^2$ error on the state and adjoint variables along the iterations for different numbers of time intervals, together with the error decay predicted by our upper bound $\widetilde{\rho}$ derived in Section \ref{sec:3}. We observe that, while the time PSM is indeed weakly scalable, its convergence rate deteriorates as the time interval length decreases. To further investigate this behavior, we show in the right column of Figure~\ref{fig:WSTest} the corresponding spectral radius $\rho(T^{\text{PS}}_N)$ for each time interval length. We remark that for moderate time intervals (e.g., $\Delta t\geq 1/4$), the spectral radius is almost constant as a function of the number of time intervals $N$. However, when $\Delta t$ decreases, the spectral radius starts to vary with respect to $N$, and only for sufficiently large $N$ it reaches the asymptotic plateau. Note that, in addition, for short time intervals, the spectral radius gets larger. In particular for $\Delta t=1/16$, the spectral radius in the case of $N=16$ is more than twice that of $N=2$. This explains the deterioration of the weak scalability for small time interval lengths. Furthermore, note that our bound $\widetilde \rho$ derived in Section~\ref{sec:matrixnorm} provides a very good asymptotic estimate of the spectral radius, and gives a very accurate prediction of the convergence behavior for large $N$.  Although the infinity norm is smaller than one for all test cases, it cannot accurately describe the convergence behavior.

\subsection{Periodic Heating-cooling process}\label{sec:numcooling}
We now use our time PSM~\eqref{eq:PSQ2-QN-1}-\eqref{eq:PSQN} to simulate a real-life application. We consider a rod (or an enclosed room) with a heat source located at its center. The thermal input is time-dependent and periodic. During each time interval $\Delta t$, the heat source power is smoothly increased from zero to its maximum power and then decreased back to zero. This type of periodic heating-cooling process arises in many industrial and engineering applications. For instance, in laser-based material processing and characterization, materials are often subjected to pulsed heating to investigate their thermal properties. Similarly, building climate control systems also operate in cyclic modes and switch heating (or cooling) power on and off to regulate indoor temperature efficiently.

To mimic heating-cooling processes, we consider a target using Gaussian functions
\[\hat y(x, t) = 10\sum_{n=1}^N \exp\left(-50\left(\left(x - \frac L2\right)^2 + \left(t - \frac{(2n-1)\Delta t}2\right)^2\right)\right).\]
Figure~\ref{fig:HeatCoolEx} illustrates on the top panel such a process for four periods ($N=4$) with each period $\Delta t=1/2$ and space length $L=1$.
\begin{figure}
\centering
\includegraphics[scale=0.3]{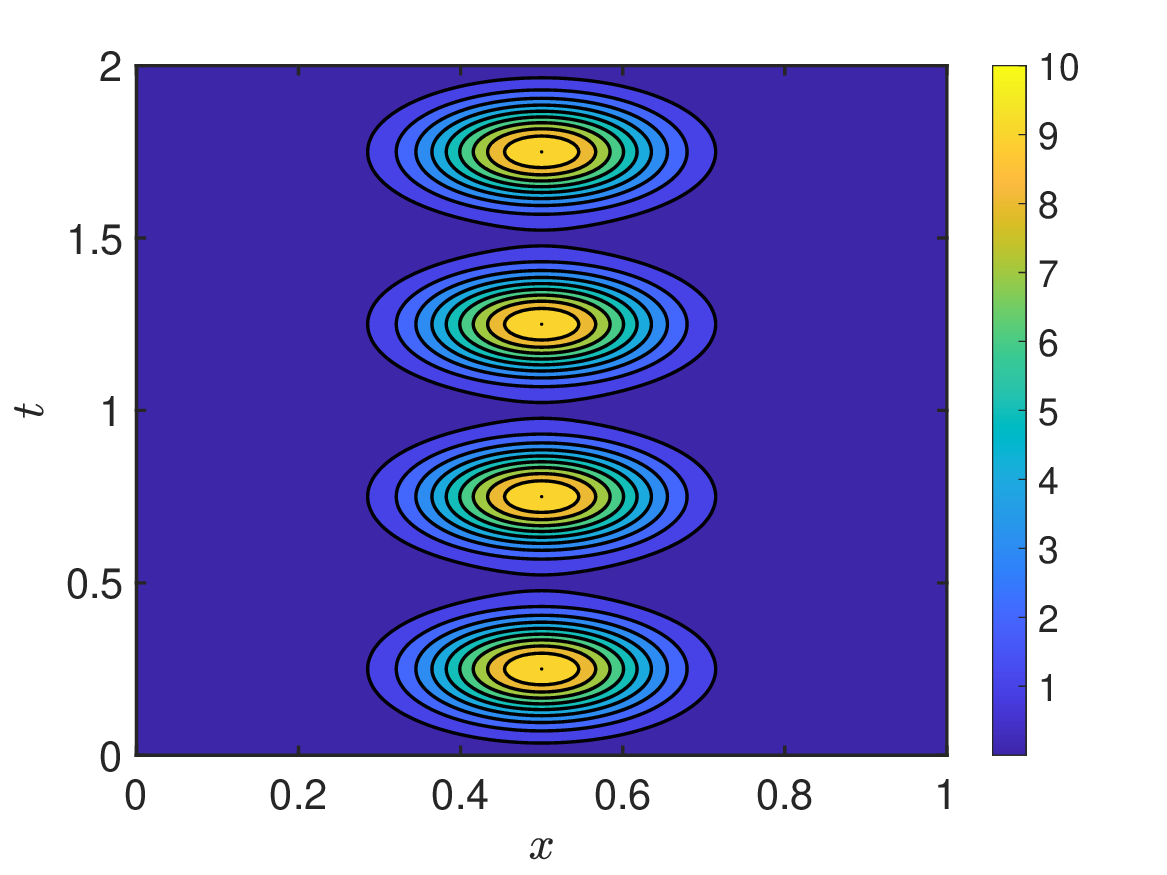}\\
\includegraphics[scale=0.3]{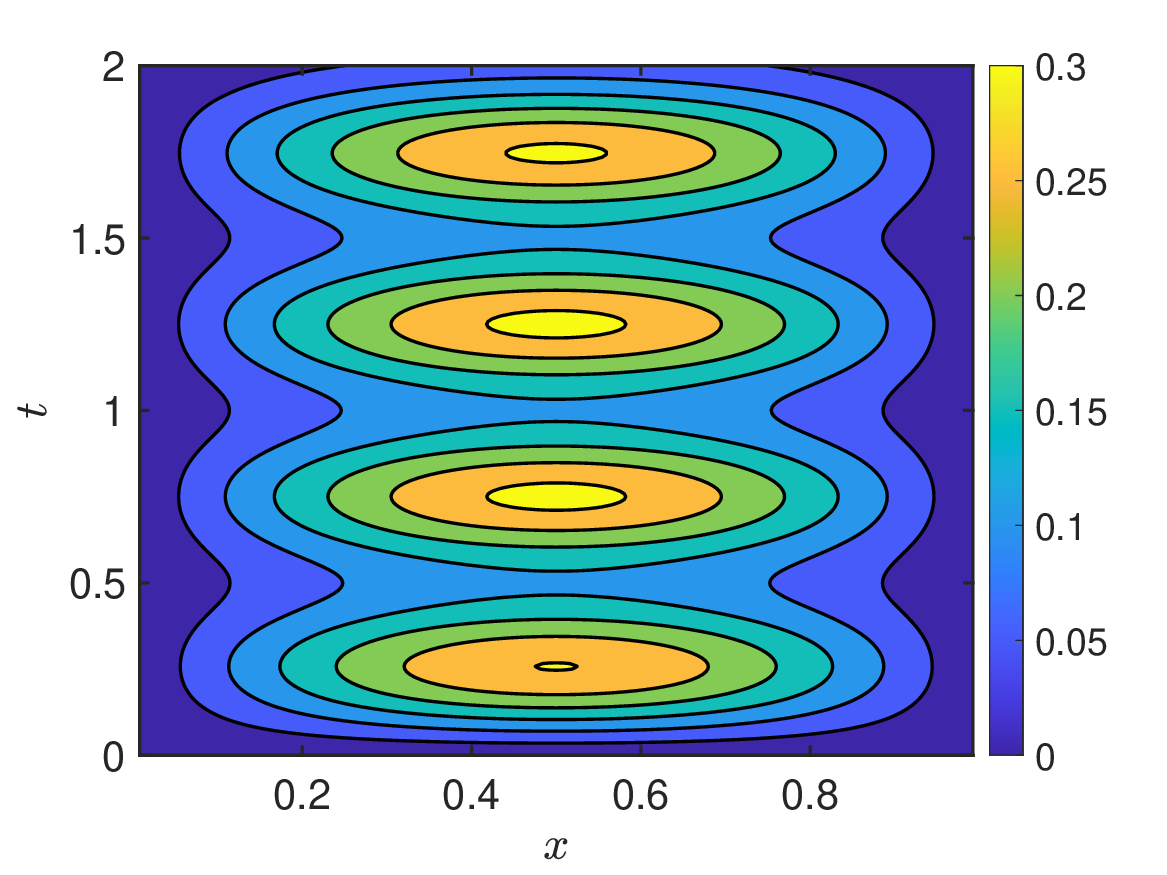}
\includegraphics[scale=0.3]{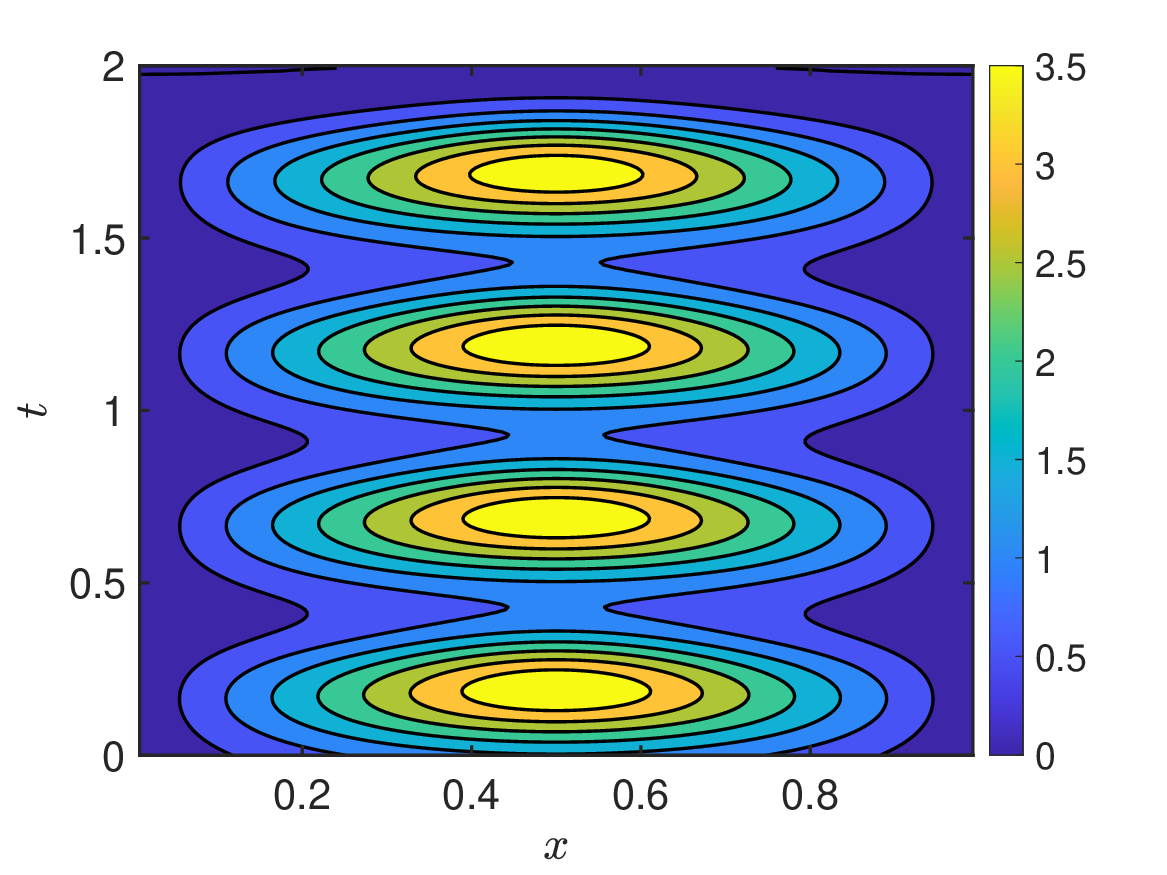}
\caption{Illustration of periodic heating-cooling process for four periods with each period $\Delta t=1/2$ and the space domain $\Omega = (0, 1)$. The top panel shows the target state $\hat{y}$, the left-bottom panel shows the state solution $y$ and the right-bottom panel shows the optimal control $u$.}
\label{fig:HeatCoolEx}
\end{figure}
As the heating-cooling process repeats periodically in time, this is a perfect example to apply our time PSM~\eqref{eq:PSQ2-QN-1}-\eqref{eq:PSQN}. We use once again the Crank--Nicolson method with $h_t = h_x = 1/128$ to discretize the reduced optimality system~\eqref{eq:reduced}. The bottom panels of Figure~\ref{fig:HeatCoolEx} present the solution $y$ and the control $u$ obtained for the four-period example with penalization parameter $\nu = 1/10$.
We observe that the optimal control $u$ (bottom right of Figure~\ref{fig:HeatCoolEx}) is also periodic, whereas the solution $y$ (bottom left of Figure~\ref{fig:HeatCoolEx}) is almost periodic and its magnitude is much smaller than the target $\hat y$ due to the choice of the penalization parameter $\nu$.

To test the scalability of our time PSM~\eqref{eq:PSQ2-QN-1}-\eqref{eq:PSQN}, we keep the same size of period $\Delta t=1/2$ and consider the number of periods varying in a test set: $N\in\{2^1, 2^2, \ldots ,2^9\}$. Figure~\ref{fig:HeatCoolWS} illustrates the $L^2$ error between the solutions of the optimality system over the entire time horizon and the solutions concatenated at each iteration. 
\begin{figure}
\centering
\includegraphics[scale = 0.25]{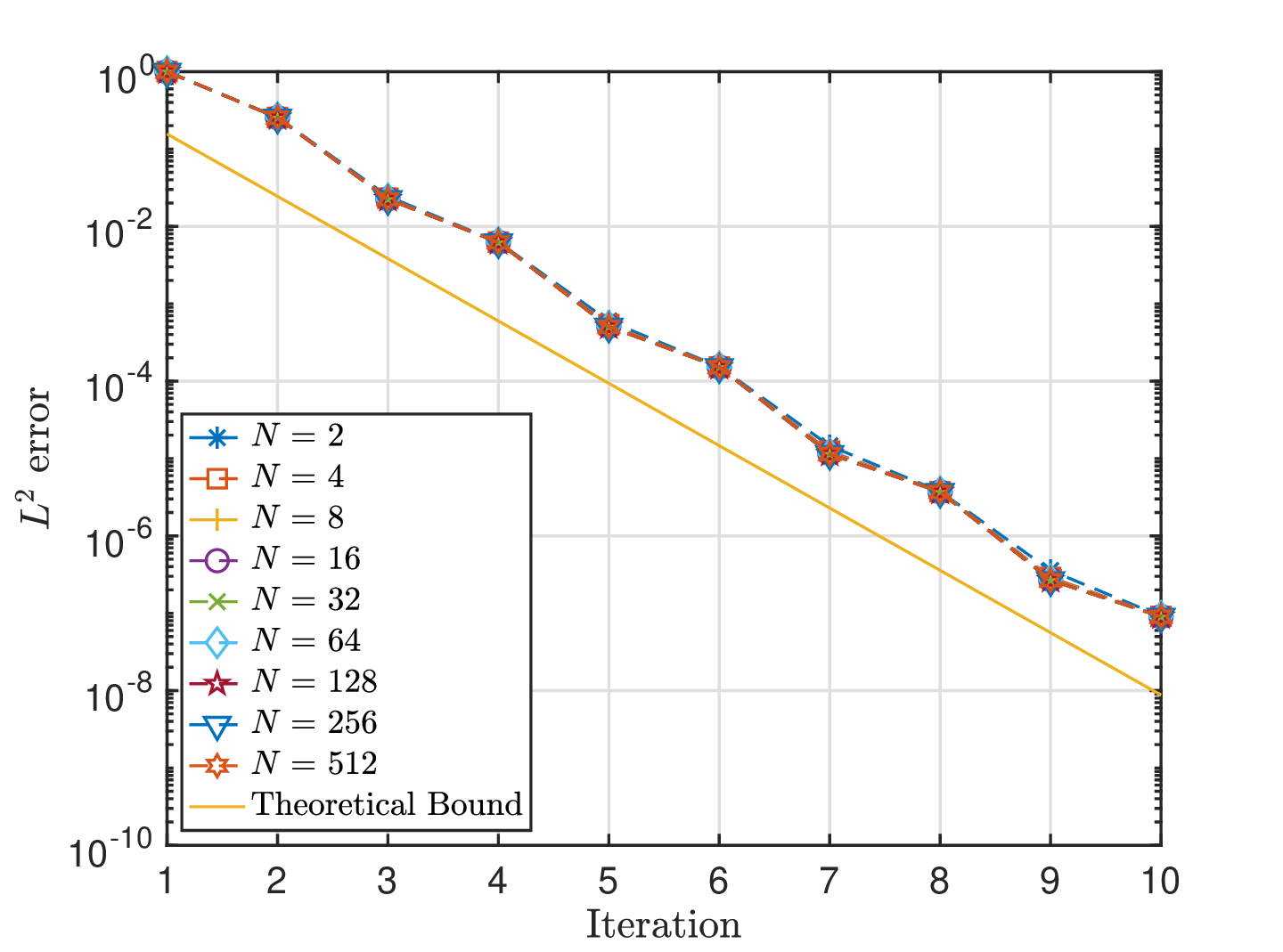}
\caption{Illustration of the $L^2$ error between solutions of the entire fully discrete optimality system and solutions concatenated at each iteration. Each period $\Delta t = 1/2$ and the number of periods $N\in\{2^1, 2^2, \ldots ,2^9\}$.}
\label{fig:HeatCoolWS}
\end{figure}
In particular for $2^9$ periods, we have 8,323,326 unknowns of both state ($y$) and adjoint ($p$) variables for the fully discrete optimality system associated with~\eqref{eq:reduced}, whereas there are 16,510 unknowns of both variables within each time period. We remark that our theoretical bound derived in Section~\ref{sec:matrixnorm} provides a very accurate prediction of the convergence behavior. Finally, we observe once again that the time PSM is weakly scalable, as the number of iterations to reach the desired tolerance is independent of the number of time intervals. Therefore, it is a promising solution strategy for industrial processing requiring parallelization on high-performance computing workstations.

\section{Conclusion}\label{sec:5}

We investigated the convergence behavior of the time PSM applied to the first-order optimality system of parabolic optimal control problems. To demonstrate the weak scalability of this method, we derived convergence estimates for the spectral radius of the iteration matrix using two different approaches. The first approach constructs a special matrix norm, under which the spectral radius of the iteration matrix is bounded by this norm, which is strictly less than one and independent of the number of time intervals. The second approach relies on block Toeplitz matrix theory. We provided a nonasymptotic result that locates all eigenvalues of the iteration matrix in a complex plane, with their moduli bounded by the same special matrix norm identified in the first approach. We also derived an asymptotic result that characterizes the convergence behavior of the time PSM as the number of time intervals tends to infinity. Numerical experiments further showed that our convergence estimates are very sharp and accurately capture the convergence behavior of the time PSM. This work provides the first theoretical framework for analyzing the weak scalability of time domain decomposition methods. Further research directions include generalizing the weak scalability analysis to other time domain decomposition methods, leveraging the analysis presented here to develop multi-level solvers, and implementing the time PSM to solve more complex parabolic optimal control problems on high-performance computing platforms. 

\section{Acknowledgements}
T.V. is a member of the INdAM-GNCS group. 


\end{document}